\documentclass[10pt,reqno]{amsart}

\usepackage{color}

\usepackage{mathrsfs}
\usepackage{amssymb}
\usepackage{enumerate}

\newtheorem{thm}{Theorem}[section]
\newtheorem{defn}[thm]{Definition}
\newtheorem{corollary}[thm]{Corollary}
\newtheorem{lemma}[thm]{Lemma}

\newtheorem{example}[thm]{Example}
\newtheorem{assumption}[thm]{Assumption}

\theoremstyle{remark}
\newtheorem{remark}[thm]{Remark}

\def\qed{{\hfill $\Box$ \bigskip}}

\def\XXint#1#2#3{{\setbox0=\hbox{$#1{#2#3}{\int}$}
\vcenter{\hbox{$#2#3$}}\kern-.5\wd0}}

\newcommand\aint{-\hspace{-0.38cm}\int}

\newcommand\cbrk{\text{$]$\kern-.15em$]$}}
\newcommand\opar{\text{\,\raise.2ex\hbox{${\scriptstyle
|}$}\kern-.34em$($}}
\newcommand\cpar{\text{$)$\kern-.34em\raise.2ex\hbox{${\scriptstyle |}$}}\,}

\def\<{\langle}
\def\>{\rangle}

\def\E{{\mathbb E}}

\newcommand\bL{\mathbb{L}}
\newcommand\bR{\mathbb{R}}

\newcommand\bP{\mathbb{P}}
\newcommand\bQ{\mathbb{Q}}
\newcommand\bS{\mathbb{S}}

\newcommand\bM{\mathbb{M}}
\newcommand\bE{\mathbb{E}}
\newcommand\bN{\mathbb{N}}

\newcommand\bZ{\mathbb{Z}}

\newcommand\fB{\mathbf{B}}

\newcommand\fR{\mathbf{R}}

\newcommand\cA{\mathcal{A}}

\newcommand\cF{\mathcal{F}}

\newcommand\cH{\mathcal{H}}
\newcommand\cI{\mathcal{I}}

\newcommand\cL{\mathcal{L}}
\newcommand\cP{\mathcal{P}}

\newcommand\cM{\mathcal{M}}
\newcommand\cT{\mathcal{T}}
\newcommand\cO{\mathcal{O}}

\newcommand\rF{\mathscr{F}}

\newcommand\rP{\mathscr{P}}

\def\R {{\mathbb R}}

\newcommand{\mysection}[1]{\section{#1}
\setcounter{equation}{0}}

\begin{document}

\title[Pure Jump Diffusion Equations]
{An $L_p$-theory for diffusion equations related to stochastic processes with non-stationary independent increment}

\author{Ildoo Kim}
\address{Center for Mathematical Challenges, Korea Institute for Advanced Study, 85 Hoegiro Dongdaemun-gu,
Seoul 130-722, Republic of Korea} \email{waldoo@kias.re.kr}
%\thanks{The first author was supported by the TJ Park Science  
%         Fellowship of POSCO TJ Park Foundation}

\author{Kyeong-Hun Kim}
\address{Department of Mathematics, Korea University, 1 Anam-dong,
Sungbuk-gu, Seoul, 136-701, Republic of Korea}
\email{kyeonghun@korea.ac.kr}
\thanks{The second and third authors were supported by Samsung Science  and Technology Foundation under Project Number SSTF-BA1401-02}

\author{Panki Kim}
\address{Department of Mathematical Sciences and Research Institute of Mathematics,
Seoul National University,
Building 27, 1 Gwanak-ro, Gwanak-gu
Seoul 151-747, Republic of Korea.}
\email{pkim@snu.ac.kr}
%\thanks{Panki Kim was supported by Samsung Science  and Technology Foundation under Project Number SSTF-BA1401-02}



\maketitle

\begin{abstract}
Let $X=(X_t)_{t \ge 0}$ be a stochastic process which has 
an (not necessarily stationary) independent 
 increment on a probability space $(\Omega, \mathbb{P})$. 
In this paper, we study the following Cauchy problem related to the stochastic process $X$: 
\begin{align}						
					\label{ab main}
\frac{\partial u}{\partial t}(t,x) = \cA(t)u(t,x) +f(t,x), \quad u(0,\cdot)=0, \quad (t,x) \in (0,T) \times \mathbf{R}^d,
\end{align}
where 
$f \in L_p( (0,T) ; L_p(\mathbf{R}^d))=L_p( (0,T) ; L_p)$ and
\begin{align*}
\cA(t)u(t,x) = \lim_{h \downarrow 0}\frac{\mathbb{E}\left[u(t,x+X_{t+h}-X_t)-u(t,x)\right]}{h}.
\end{align*}
We provide a sufficient condition on $X$ (see Assumptions \ref{as 2} and \ref{as})  to guarantee the unique solvability of  equation (\ref{ab main}) in $L_p\left( [0,T] ; H^\phi_{p}\right)$, where $H^\phi_{p}$ is a $\phi$-potential 
space on $\mathbf{R}^d$
 (see Definition \ref{psi sp}). Furthermore we show that for this solution,
\begin{align*}
\| u\|_{L_p\left( [0,T] ; H^\phi_{p}\right)} \leq N \|f\|_{L_p\left( [0,T] ; L_p\right)},
\end{align*}
where $N$ is independent of $u$ and $f$. 
\end{abstract}

\mysection{Introduction}

Roughly speaking, the second-order diffusion  equations  describe  the motion of  diffusion particles   moving according to a 
law of stochastic process driven by a Brownian motion. Such  equations are not suitable for  natural 
%phenomenon
phenomena
 with jumps, and accordingly 
 there has been growing interest in equations with non-local operators   related to
pure jump processes owing to their applications in various models in physics, economics,
engineering and many others involving long-range interactions.

If the non-local operators are close to fractional Laplacian operator, then  there are considerable regularity results.   See e.g. \cite{bass1}, \cite{bass2}, \cite{dkim2}, \cite{KSV8} and \cite{S} for  the Harnack inequality and H\"older estimates.
Regarding $L_p$-regularity theory,    H. Dong and D. Kim \cite{DongKim2012}   obtained a sharp $L_p$-estimate for  the  nonlocal elliptic equation
\begin{equation}
 \label{eq 11.16.1}
Lu -\lambda u = f \quad \text{in}~\fR^d,
\end{equation}
where
$$
Lu(x) = \int_{\fR^d} \left( u(x+y)-u(x) - y \cdot \nabla u(x) \chi^{(\alpha)}(y) \right)\frac{a(y)}{|y|^{d+\alpha}} dy,
$$
$\alpha \in (0,2)$, $\chi^{(\alpha)}$ is a certain indicator function,
%and $a(y)$ is measurable 
and $a(y)$ is a measurable function
with positive lower and upper bounds, that is, there exists $\delta>0$
\begin{align}
					\label{e 1107 1}
\frac{\delta}{|y|^{d+\alpha}} \leq \frac{a(y)}{|y|^{d+\alpha}} \leq \frac{\delta^{-1}}{|y|^{d+\alpha}} \quad \forall y \in \fR^d.
\end{align}
Observe that, since $ \alpha \in (0,2)$, 
%$\frac{a(y)}{|y|^{d+\alpha}} dy$
${a(y)}{|y|^{-d-\alpha}} dy$
 is a L\'evy measure, i.e.
$$
\int_{\fR^d} \left( 1 \wedge |y|^2 \right) \frac{a(y)}{|y|^{d+
\alpha}} dy < \infty,
$$      
and 
$\frac{C(\alpha, d)}{|y|^{d+\alpha}} dy$
 is the L\'evy measure of the rotationally invariant $\alpha$-stable process.  
%X. Zhang \cite{zhang2013lp} 
In \cite{zhang2013lp}, X. Zhang
introduced  a generalization of  (\ref{e 1107 1}). 
More precisely, he handled the Cauchy problem in $L_p$-space with the L\'evy measure $\nu(dy)$ (instead of $a(y)|y|^{-d-\alpha}dy$) with the condition 
$\nu_1^\alpha(dy)  \leq  \nu(dy) \leq \nu_2^\alpha(dy)$,
where  $\nu_i^{(\alpha)}$,  $i=1,2$, are the L\'evy measures of two $\alpha$-stable processes taking the form
\begin{align*}
\nu_i^{(\alpha)}(B) := \int_{\bS^{d-1}} \Big( \int_0^\infty \frac{1_{B}(r\theta)dr}{r^{1+\alpha}} \Big) \Sigma_{i}(d\theta),
\end{align*}
$\bS^{d-1}$ is the unit sphere in $\fR^d$ and $\Sigma_i$ is a measure on $\bS^{d-1}$.  We also refer to a recent result  \cite{mikulevicius2016lp}, where $L_p$-theory is presented for the elliptic and parabolic equations
$$
Lu-\lambda u=g, \quad  \quad \partial_t u=Lu-\lambda u+f, 
$$
on $\fR^d$ and $[0,T]\times \fR^d$ respectively. Here
\begin{align*}
L \phi(x) = L^\pi \phi (x) = \int_{\fR^d} \left(\phi(x+y) - \phi(x) -  \chi ^{(\pi)}(y) \cdot \nabla \phi(x) \right) 
\pi (dy), 
\end{align*}
and $\pi$ is supposed to satisfy a certain scaling property,  
%called assumption $\mathbf{D}(\kappa,\ell)$.
which is called assumption $\mathbf{D}(\kappa,\ell)$ in \cite{mikulevicius2016lp}.

In this article we prove the unique solvability of  diffusion equation (\ref{main eqn}) with the generator of stochastic processes beyond  L\'evy processes.
In particular, we focus on  diffusion equations with generators of stochastic processes with non-stationary independent increments.  For instance, our stochastic processes $X_t$ can be of type $X_t=\int^t_0 a(s)dY_s$, where $Y_t$ is a subordinate Brownian motion, and $X_t$ can also be an additive process.  See Section 2.2 for more concrete examples.   We adopt  $\phi$-potential space (see  \cite{farkas2001function}) for the space of  solutions.  This is because  our operators are far away from $\alpha$-stable process and the classical Bessel potential space does not fit 
as a solution space.

We emphasize that  even if the stochastic process $X_t$ is a L\'evy process, our result cannot be covered by above results.  
For instance, an example related to Subordinate Brownian motions is given in \cite[Example 2.1 and Remark 2]{mikulevicius2016lp}.
In this example, there are conditions on 
%$\delta_1$ and $\delta_2$
weak scaling constants $\delta_1$ and $\delta_2$  such as $2 \delta_1 > 1$ and $2 \delta_1 >\delta_2$. 
However, we do not need this relation in our results (see Example \ref{ex1}).

Next we give a few remarks on our methods.
%Due to the non-local property, classical 
Due to the non-local property of our operators, classical 
perturbation arguments  are not available.  
Nonetheless, fortunately, our operators are still 
%psedu-differential operators.
pseudo-differential  operators.
 If the symbols of 
 %pseudo-differentiable operators
 pseudo-differential  operators
  are smooth enough  then one can use classical tools 
from Harmonic and  Fourier analysis. 
However, if the moments of the given  process are not finite, the symbol of the generator of the  
 process
 loses the smoothness property.  We 
 overcome this difficulty using a probabilistic technique  together with analytic tools.
Technically our approach does not rely on the well-developed one-parameter semi-group theory 
since increments of 
%the
our stochastic processes are not stationary.

The article is organized as follows. In section 2, we present our main result (Theorem \ref {main thm}), $L_p$-theory of  PDEs with generators of  non-stationary  independent increment processes. 
In Section 3, we introduce a version of singular integral theory which  fits  our equations. 
In Section 4, we prove a maximal $L_p$-regularity theory for a class of pseduo differential operators. The result of this section is used to prove our main result when the symbol of the operator is smooth. Section 5 contains the proof of our main theorem, and 
finally in Appendix we prove a version of the Fefferman-Stein theorem.

%We finish the introduction with  notation used in the article.
We finish the introduction by introducing   notations we will  use in the article.
$\bN$ and $\bZ$ denote the natural number 
%ystem
system
and the integer number system, respectively.  Denote $\bZ_+ :=\{ k \in \bZ ; k  \geq 0\}$. 
As usual $\fR^{d}$
stands for the Euclidean space of points $x=(x^{1},...,x^{d})$.
 For $j=1,...,n$, multi-indices $\alpha=(\alpha_{1},...,\alpha_{n})$,
$\alpha_{j}\in \bZ_+$, and functions $u(x)$ we set
$$
u_{x^{j}}=\frac{\partial u}{\partial x^{j}}=D_{j}u,\quad
D^{\alpha}u=D_{1}^{\alpha_{1}}\cdot...\cdot D^{\alpha_{d}}_{d}u,
\quad  \nabla u=(u_{x^1}, u_{x^2}, \cdots, u_{x^d}).
$$
%{\color{magenta}  We also use the notation $D^m_x$ for a partial derivative of order $m$ with respect to $x$.}
We also use the notations $D^m_x$ (and $D^\alpha_x$, respectively) for a partial derivative of order $m$ (of multi-index  $\alpha$, respectively)  with respect to $x$.
$C(\fR^d)$ denotes the space of bounded continuous functions on $\fR^d$.
For $n \in \bN$, we write $u \in C^n(\fR^d)$  if $u$ is $n$-times continuously differentiable in $\fR^d$, and $\sup_{x \in \bR^d, |\alpha|\leq n} |D^\alpha u|< \infty$. Simply we put $C^n := C^n(\fR^d)$. 
For $p \in [1,\infty)$, 
a normed space $F$ with norm $ \| \cdot \|_F$
and a  measure space $(X,\mathcal{M},\mu)$, $L_{p}(X,\cM,\mu;F)$
denotes the space of all $F$-valued $\mathcal{M}^{\mu}$-measurable functions
$u$ so that
\[
\left\Vert u\right\Vert _{L_{p}(X,\cM,\mu;F)}:=\left(\int_{X}\left\Vert u(x)\right\Vert _{F}^{p}\mu(dx)\right)^{1/p}<\infty,
\]
where $\mathcal{M}^{\mu}$ denotes the completion of $\cM$ with respect to the measure $\mu$.

For $p=\infty$, we write $u \in L_{\infty}(X,\cM,\mu;F)$ iff
$$
\|u\|_{L_{\infty}(X,\cM,\mu;F)}
:= \inf\left\{ \nu \geq 0 : \mu( \{ x: \|u(x)\|_F > \nu\})=0\right\} <\infty.
$$
If there is no confusion for the given measure and $\sigma$-algebra, we usually omit the measure and the $\sigma$-algebra.
In particular, for a domain $U \subset \fR^d$ we denote $L_p(U) = L_p(U,\cL, \ell ;\fR)$,
where $\cL$ is the Lebesgue measurable sets,  and $\ell$ is 
the  Lebesgue measure in $\fR^d$.
We use the notation $N$ to denote a generic constant which may change from line to line. 
%If we write $N=N(a,b,\cdots)$, this means that the
%constant $N$ depends only on $a,b,\cdots$.
While, throughout this paper the constants
 $N_j$,  $j=0, 1,\dots,$ will be fixed.
We use $N=N(a,b,\cdots)$ to indicate
a positive constant that depends on the parameters $a,b,\cdots$.

We use  ``$:=$" or ``$=:$" to denote a definition. $\lfloor a \rfloor$ is the biggest integer which is less than or equal to $a$.
By $\cF$ and $\cF^{-1}$ we denote the d-dimensional Fourier transform and the inverse Fourier transform, respectively. That is,
$\cF[f](\xi) := \int_{\fR^{d}} e^{-i x \cdot \xi} f(x) dx$ and $\cF^{-1}[f](x) := \frac{1}{(2\pi)^d}\int_{\fR^{d}} e^{ i\xi \cdot x} f(\xi) d\xi$
where $i$ is the imaginary number, i.e. $i^2=-1$.
%For a Lebesgue measurable set $A\subset \fR^d$, 
%we use  $1_A(x)$ to denote  the indicator of $A$.
We use $1_B$ to denote  the indicator of a set $B$.
For a Lebesgue measurable set $A$, 
we use $|A|$ to denote its Lebesgue measure.
 For a complex number $z$, $\Re[z]$ is the real part of $z$
and $\bar z$ is the complex conjugate of $z$.
%new
For  a function space $\cH(U)$
on an open set $U$ in $\R^d$, we let   
$\cH_c(U):=\{f\in\cH(U): f \mbox{ has  compact support}\},$
$\cH_0(U):=\{f\in\cH(U): f \mbox{ vanishes at infinity}\}$.
%endnew

\mysection{PDEs with generators of 
 independent  increment process}

\subsection{Assumptions}
 Let $T<\infty$.  Every stochastic process considered in this article is 
$\fR^d$-valued.
Recall that a measure $\mu$ on $\fR^d$ is a  L\'evy measure 
%iff
if
\begin{align*}
\mu(\{0\})=0 \quad \text{and} \quad \int_{\fR^d} \left(1 \wedge |x|^2 \right)\mu(dx) <\infty.
\end{align*}
\begin{assumption}
							\label{as 2}
(i) The stochastic process $X$ has a 
 pure jump component, that is  there exist two independent stochastic process $X^1$ and $X^2$ such that for all $t \geq s \geq 0$,
$X_t-X_s$ and  $X^1_t-X^1_s+X^2_t-X^2_s$  have same distributions  and 
\begin{align*}
&\bE e^{i \xi \cdot (X^1_t-X^1_s)} = \exp \left( \int_s^t \int_{\fR^d}\left( e^{i\xi \cdot x} -1 -i\xi \cdot x 1_{|x| <1} \right)\mu_r(dx) dr \right),
\end{align*}
where $\mu_r$ is a L\'evy measure for each $r \in [0,\infty)$.

(ii)  The paths of $X^2$ are locally bounded $(a.s.)$. i.e., 
$$
\bP \Big(\sup_{t \in [a,b]}|X^2_t| < \infty\Big)=1. \quad \text{ for all } \, 0<a<b< \infty.
$$				

\end{assumption}

Denote 
$$
\Psi_{X^1}(t,\xi)
:=\int_{\fR^d}\left( e^{i\xi \cdot x} -1 -i\xi \cdot x 1_{|x| <1} \right)\mu_t(dx)
$$
and
$$
\Phi_{X^1}(s,t,\xi):=\int_s^t  \Psi_{X^1}(r,\xi)\,dr.
$$

\begin{assumption}
					\label{as}
Let $d_0:= \left\lfloor \frac{d}{2} \right\rfloor+1$.
(i) There exists 
%a function
a complex-valued function 
$\Psi_X(t,\xi)$ on $[0,\infty)  \times \fR^d$ so that 
for all $t \geq s>0$ and $\xi \in \fR^d$
\begin{align*}
\bE e^{i \xi \cdot (X_t-X_s)}
=\exp \left( \int_s^t \Psi_X(r,\xi)dr\right)
=:\exp \left( \Phi_X(s,t,\xi)\right),
\end{align*}
and furthermore $\Psi_{X}$ and $\Phi_{X}$ satisfy the followings: 
\begin{itemize}
\item For each $\xi$, $\Psi_X(t,\xi)$ is locally integrable with respect to $t$ on $[0, T)$,
i.e. $\Psi_X(\cdot , \xi) \in L_1([0,t))$ for all $t \in [0,T)$.

\item 
$\xi \to \exp\left(\Phi_X(s,t,\xi)\right)$ and $\xi \to \Psi_X(t,\xi)\cdot \exp\left(\Phi_X(s,t,\xi)\right)$
 are locally bounded and have at most polynomial growth at infinity with respect to $\xi$ uniformly for $0<s<t<T$, i.e. there exists a $N>0$ so that
 \begin{align}
 \label{e:PhiPsi_upper}
 \sup_{0<s<t<T} \left(|\exp\left(\Phi_X(s,t,\xi)\right)| + |\Psi_X(t,\xi)\cdot \exp\left(\Phi_X(s,t,\xi)\right)| \right)
\leq N \left(1+ |\xi| \right)^N.
 \end{align}
\end{itemize}

(ii) There exists a nondecreasing function $\phi(\lambda) : (0,\infty) \mapsto (0,\infty) $ and positive constants $\delta_k$ and $N_j$ ($k=1,2,3$ and $j=1,2,3,4$) such that  
\begin{itemize}
\item   for all $\xi \neq 0$
\begin{align}
						\label{main as 3}
      \Re [-\Psi_{X^1}(t,\xi)]  \geq  \delta_1 \phi(|\xi|^2),
\end{align}
\item   for all $\xi \neq 0$ and multi-index $|\alpha| \leq  
d_0$,
\begin{align}
						\label{main as 4}
 |D^\alpha_{\xi} \Psi_{X^1}(t,\xi)| \leq  N_1 |\phi(|\xi|^2)| |\xi|^{-|\alpha|},
\end{align}

\item $ \delta_3 \geq \delta_2 > 0$ and  for any 
$\lambda_2 \geq \lambda_1>0$,
\begin{align}
						\label{main as}
N_2 \left(\frac{\lambda_2}{\lambda_1}\right)^{\delta_2} \leq  \frac{\phi(\lambda_2)}{\phi(\lambda_1)}  
\leq N_3 \left(\frac{\lambda_2}{\lambda_1} \right)^{\delta_3},
\end{align}

\item 
 for 
all $\lambda \in (0,\infty)$ and  natural number $n \leq d_0$,
\begin{align}
						\label{main as 2}
|D^n\phi(\lambda)| \leq N_4 \lambda^{-n} \phi(\lambda).
\end{align}
\end{itemize}
\end{assumption}

\begin{remark}
If  $X^2=0$, then Assumption \ref{as}
(ii) implies (i). 
\end{remark}

\subsection{Examples}

To introduce  examples  satisfying above assumptions, we recall some definitions and facts on subordinate Brownian motion.   
A function  $\phi : (0,\infty) \to (0,\infty)$  is called  a Bernstein function  with $\phi(0+)=0$ if $\phi$ 
 has a representation that
\begin{align}
						\label{bern repre}
\phi(\lambda) = b \lambda + \int_0^\infty (1- e^{-\lambda t}) \mu(dt),
\end{align}
where 
$b \geq 0$
 and $\mu$ is a measure on $(0,\infty)$ satisfying $\int_0^\infty (1 \wedge t) \mu(dt) < \infty$. 
Then it is well-known that $($e.g. \cite[Chapter 3]{schilling2012bernstein} and \cite[Lemma 3.2]{kim2013parabolic}$)$
\begin{align}
						\label{914 1}
\frac{\phi(\lambda_2)}{\phi(\lambda_1)} \leq \frac{\lambda_2}{\lambda_1}, \qquad 0<\lambda_1 < \lambda _2.
\end{align}
and for any nonnegative integer $n$,
\begin{align}
\label{e:dBFb}
\lambda^n |D^n\phi(\lambda)| \leq N(n) \phi(\lambda), \qquad \forall \lambda >0.
\end{align}
Thus $\phi$ satisfies (\ref{main as 2}).%First we give two examples related to subordinate Brownian motions.

 Let
 $S=(S_t)_{t\ge 0}$ be a subordinator  (i.e. an
increasing L\'evy process satisfying $S_0=0$), then there is a Bernstein function 
$\phi$ with $\phi(0+)=0$
such that $\bE e^{-\lambda S_t}=e^{-t\phi(\lambda)}$.  
Let $W=(W_t)_{t\ge 0}$ be a Brownian motion in $\R^d$,  
i.e. $\E\left[e^{i\xi \cdot W_t}\right]=e^{-t{|\xi|^2}}, \xi\in \R^d, t>0$,
which is independent of $S_t$. 
Then
     $Y_t:=W_{S_t}$,
     called the subordinate Brownian motion (SBM),   is a rotationally invariant L\'evy process in $\R^d$
with characteristic exponent $\phi(|\xi|^2)$, and by L\'evy-Khintchine theorem, 
\begin{equation}
 \label{levy1}
\bE [ e^{i \xi \cdot 
Y_t}] 
= e^{-t\phi(|\xi|^2)}
= \exp\left(t\int_{\fR^d}\left( e^{i\xi \cdot x} -1 -i\xi \cdot x 1_{|x| <1} \right)J(x)dx \right),
\end{equation}
where $J(x) = j(|x|)$ and 
$$
j(r)= \int_0^\infty (4\pi t)^{-d/2} e^{-r^2/(4t)} 
\mu(dt). 
$$

\begin{example}[Integral with respect to SBM]
					\label{ex2}

Assume 
that  the Bernstein function $\phi$ satisfies the following  
weak-scaling conditions:
\begin{itemize}
\item  There exist constants $0< \delta_1 \leq \delta_2  <1$ and $a_1 >0$ such that
\begin{align}\label{e:wsc}
a_1\left(\frac{\lambda_2}{\lambda_1}\right)^{\delta_1} 
\leq \frac{\phi(\lambda_2)}{\phi(\lambda_1)}
\leq a_2\left(\frac{\lambda_2}{\lambda_1}\right)^{\delta_2},
\qquad  0< \lambda_1 \leq \lambda_2 <\infty
\end{align}
\end{itemize}
Note that $\phi$ satisfies (\ref{main as}) and (\ref{main as 2}).

Let $\sigma : (0,\infty) \mapsto \fR$ be a bounded measurable function such that for all $t \in (0,\infty)$, $|\sigma(t)|\in (\delta, \delta^{-1})$ for some $\delta \in (0,1)$.  
Recall  $ Y_t=B_{S_t}$ and define 
$$
X^1_t:=\int_0^t \sigma(s) dY_s.
$$
Then
\begin{align*}
\bE e^{i \xi \cdot (X_t-X_s)} 
= \exp \left[ \int_s^t \left( \int_{\fR^d}\left( e^{i\xi \cdot x} -1 -i\xi \cdot x 1_{|x| <1} \right)\mu_{r}(dx) \right)dr \right],
\end{align*}
where $\mu_r(B) =\int_{\fR^d} 1_B(\sigma(r) x) J(x)dx$.  
Thus 
denoting 
$$
\Psi_{X^1}(t,\xi)=\int_{\fR^d}\left( e^{i\xi \cdot x} -1 -i\xi \cdot x 1_{|x| <1} \right)\mu_{t}(dx),
$$
we see that $X^1_t$ has a pure jump independent increment. 
Moreover, 
Due to the properties of the density function of Brownian motion and (\ref{bern repre}), 
\begin{align*}
-\Psi_{X^1}(r,\xi)
&=-\int_{\fR^d}\left( e^{i\xi \cdot x} -1 -i\xi \cdot x 1_{|x| <1} \right)\mu_{r}(dx)   \\
&= \lim_{\varepsilon \downarrow 0} \int_0^\infty \int_{|x| > \varepsilon } \left( 1-e^{i\xi \cdot x}  \right) (4\pi t \sigma(r)^2)^{-d/2} e^{ - \frac{|x|^2}{4t \sigma(r) ^2}} dx \mu(dt) \\
&= \int_0^\infty  \left( 1- e^{-t | \sigma(r) \xi|^2}  \right) \mu(dt) \\
&= \phi\left(    (\sigma(r) |\xi|)^2 \right).
\end{align*}
Therefore, since the Bernstein function $\phi$ satisfies (\ref{e:dBFb}) and $|\sigma(r)|$ has  positive upper and lower bounds, one can easily check that  (\ref{main as 3}) and (\ref{main as 4}) hold.

% by \cite[Theorem 3.4]{kim2014global}, there exists a constant $N>0$ such that
%$$
%N\frac{\phi(|x|^{-2})}{|x|^d} \leq J(x) \leq N^{-1}\frac{\phi(|x|^{-2})}{|x|^d} \qquad \forall x \in \fR^d \setminus \{0\}
%$$
%and thus combining with (\ref{e:wsc}), we have
%$$
%N|\sigma(t)|^{d +\delta} J(x) \leq J(\sigma^{-1}(t)x) \leq N^{-1}|\sigma(t)|^{d+\bar \delta} J(x),
%$$
%where $N$, $\delta$, and $\bar \delta$ only depend only on $a_i$ and $\delta_k$ $(i=1,2,3,4)$. 

\end{example}

The following well-known examples of subordinators satisfy \eqref{e:wsc}:
\begin{itemize}
\item[1.] Stable subordinator: $\phi(\lambda)=\lambda^{\alpha}$, $0<\alpha<1$, with  $\delta=1-\alpha$.
\item[2.] Sum of two stable subordinators: $\phi(\lambda)=\lambda^\beta + \lambda^\alpha$, $0<\beta<\alpha<1$, with $\delta=1-\alpha$.
\item[3.] Stable with logarithmic correction: $\phi(\lambda)=\lambda^\alpha(\log(1+\lambda))^\beta$, $0<\alpha<1$,
$0<\beta < 1-\alpha$, with $\delta=1-\alpha-\epsilon$ for every $\epsilon >0$.
\item[4.]  Stable with logarithmic correction: $\phi(\lambda)=\lambda^\alpha(\log(1+\lambda))^{-\beta}$, $0<\alpha<1$,
$0<\beta < \alpha$, with $\delta=1-\alpha$.
\end{itemize}

In the next example the condition on $\phi$ is weakened.

\begin{example}[Additive process]
					\label{ex1}
Let $X^1_t=B_{S_t}$ be   a Subordinate Brownian motion.  Assume the the Laplace exponent of $S_t$ satisfies  the following 
condition:
\begin{itemize}
\item (H) : There exist constants $0< \delta_1   \leq 1$ and $a_1 >0$ such that
\begin{align*}
a_1\left(\frac{\lambda_2}{\lambda_1}\right)^{\delta_1} \leq \frac{\phi(\lambda_2)}{\phi(\lambda_1)}, 
\qquad  0<  \lambda_1 \leq \lambda_2 <\infty
\end{align*}
\end{itemize}
Then combining (\ref{914 1}) and (H), we have
\begin{align}
						\label{915 1}
a_1\left(\frac{\lambda_2}{\lambda_1}\right)^{\delta_1} \leq \frac{\phi(\lambda_2)}{\phi(\lambda_1)} 
\leq \left(\frac{\lambda_2}{\lambda_1}\right) \quad  0<\lambda_1 \leq \lambda_2<\infty.
\end{align}
Thus $\phi$ satisfies (\ref{main as}).

Let $a(t): (0,\infty) \mapsto (0,\infty)$ be a function  which is bounded from both above and below.
Define $\Psi_{X^1}(t,\xi)= -a(t)\phi(|\xi|^2)$. 
Then obviously 
due to (\ref{e:dBFb}),
$\Psi_{X^1}$ satisfies (\ref{main as 3}) and (\ref{main as 4}).
Moreover, there exists a additive process $X^1_t$ $($see \cite[Theorem 9.8(ii) and Theorem 11.5]{sato1999levy}$)$ such that
for all $t >0$ and $\xi \in \fR^d$
\begin{align*}
\bE e^{i \xi \cdot X^1_t} 
= \exp \left[ \int_0^t \Psi_{X^1}(s,\xi) ds\right],
\end{align*}
which is because $\phi(|\xi|^2)$ has  the representation (\ref{levy1}). 
\end{example}

The following well-known examples of subordinators satisfy \eqref{915 1} but do not  satisfy \eqref{e:wsc}:
\begin{itemize}
\item[1.] Relativistic stable subordinator: $\phi(\lambda)=(\lambda+m^{1/\alpha})^{\alpha}-m$, $0<\alpha<1$ and $m>0$, with $\delta=1-\alpha$. 
\item[2.]  $$\phi(\lambda)=\frac{\lambda}{\log(1+\lambda^{\beta/2})},  \quad \text{ where } \beta\in (0,2).$$ 
\end{itemize}

In the following example, we show that locally homogeneous additive process satisfies our assumption on $X^2$.

\begin{example}
						\label{x2 example}
Let $X_t^2$ be an additive process. Then by \cite[Theorem 9.8]{sato1999levy}, there exists a triple $(a(t),A(t),\mu_t)$ so that
\begin{align}
					\label{915 2}
\bE e^{i \xi \cdot X^2_t} 
= \exp \left( ia(t) \cdot \xi  -\frac{1}{2} (A(t)\xi,\xi) 
+\int_{\fR^d}\left( e^{i\xi \cdot x} -1 -i\xi \cdot x 1_{|x| <1} \right)\mu_t(dx) \right),
\end{align}
where $t \in [0,\infty)$, $\xi \in \fR^d$, $a(t) \in \fR^d$, $A(t)$ is a nonnegative symmetric matrix, and $\mu_t$ is a 
L\'evy measure for each $t \in [0,\infty)$, and $A(t), \mu_t$ are nondecreasing.
If  $a_t$, $A_t$, $\mu_t$ are absolutely continuous with respect to $dt$, say $a_t = \int_0^t a^*(s)ds$, 
$A(t)=\int_0^tA^*(s)ds$, $\mu_t=\int^t_0 \mu^*(s)ds$, then 
from (\ref{915 2}) we have
\begin{align*}
&\bE e^{i \xi \cdot X^2_t}  \\
&= \exp \left( \int_0^t \left(ia^*(s) \cdot \xi  -\frac{1}{2} (A^*(s)\xi,\xi) 
+ \int_{\fR^d}\left( e^{i\xi \cdot x} -1 -i\xi \cdot x 1_{|x| <1} \right)\mu^*_s(dx) \right)ds\right) \\
&=: \exp \left( \int_0^t \Psi_{X^2}(s,\xi)ds \right).
\end{align*}
Due to the absolute continuity assumption, it is obvious that
$\Psi_{X^2}$ is locally integrable with respect to $t$ for each $\xi \in \fR^d$.
Moreover, for any $0<s<t<T$,
\begin{align*}
\left|\exp\left(\int_s^t\Psi_{X^2}(r,\xi)dr  \right)  \right|
=\left|\bE e^{i \xi \cdot X^2_t}  \right|
\leq 1
\end{align*}
and
\begin{align*}
&\leq  |\Psi_{X^2}(t,\xi)|   \\
&\leq \left(\sup_{t \leq T} |A(t)| |\xi|^2 + N(1+|\xi|^2) \int_0^T \int_{\fR^d} (1+|x|^2) \mu^\ast(t)(dx) dt  \right) \\
&\leq N(T)\left(1+|\xi|^2\right),
\end{align*}
which implies  both $\exp\left(\int_s^t\Psi_{X^2}(r,\xi)dr  \right)$ and $\Psi_{X^2}(t,\xi)\exp\left(\int_s^t\Psi_{X^2}(r,\xi)dr  \right)$ are locally bounded and have a polynomial growth at infinity with respect to $\xi$ uniformly for $0<s<t<T$.
Moreover obviously paths of $X^2$ are locally bounded $(a.s.)$ since $X^2$ is a  c\'adl\'ag process.

Let $X_t^1$ be the process handled in Example \ref{ex1} or Example \ref{ex2}. 
By considering product of  probability spaces, we may assume $X_t^1$ and $X_t^2$ are independent without loss of generality. 
Set $X_t = X^1_t +X^2_t$. Then $X_t$ satisfies Assumption \ref{as}. 
\end{example}

At first glance, conditions (\ref{main as 3}) and  (\ref{main as 4})  seem to be  quite strong 
since  for each $t$ $ \Psi_{X^1}(t, \xi)$ has to be smooth in $\mathbf{R}^d\setminus \{0\}$  with respect to $\xi$.
However, these conditions are  imposed only on the symbol $\Psi_{X^1}$. 
We would like to emphasize that our symbol $\Psi_{X}(t,\xi)$, which is the sum of $\Psi_{X^1(t,\xi)}$ and $\Psi_{X^2}(t,\xi):=\Psi_{X}(t,\xi) - \Psi_{X^1}(t,\xi)$, does not have to be  smooth.  We give a  concrete simple example below.
\begin{example}
					\label{new exam}
Let $\alpha \in (0,2)$ and $a(t,x)$ be a  positive   measurable function on $(0,\infty) \times \fR^d$ and assume that $a(t,x)$ is bounded from both above and below, i.e. there exists a  constant $c \in (0,1)$ such that
$$
c \leq a(t,x) \leq c^{-1} \qquad  \forall (t,x) \in (0,\infty) \times \fR^d. 
$$
Using $a(t,x)$ and $c$, we define the following  L\'evy measures  for each $t>0$;
$$
\mu_t(dx):=a(t,x)  |x|^{-d-\alpha}dx, \quad 
\mu^1_t(dx):=\frac{c}{2} |x|^{-d-\alpha}dx
$$
and
$$
\mu^2_t(dx):= \left(a(t,x) -\frac{c}{2} \right)  |x|^{-d-\alpha}dx.
$$
Define 
\begin{align*}
\Psi_X(t,\xi)
:=\int_{\fR^d} \left( e^{i\xi \cdot x} -1 -i\xi \cdot x 1_{|x| <1} \right)   \mu_t(dx).
\end{align*}
and $(k=1,2)$
\begin{align*}
\Psi_{X^k}(t,\xi)
:=\int_{\fR^d} \left( e^{i\xi \cdot x} -1 -i\xi \cdot x 1_{|x| <1} \right)   \mu^k_t(dx).
\end{align*}
Observe that by the change of variables
\begin{align}
					\label{n smooth s}
\Psi_X(t,\xi)
&=\int_{\fR^d} \left( e^{i\xi \cdot x} -1 -i\xi \cdot x 1_{|x| <1} \right)   \mu_t(dx)  \\
&=-|\xi|^{\alpha} \int_{\fR^d} \left(  1-\cos (x^1 )  \right)   a\left(t,\frac{ O_{\xi} x}{|\xi|} \right) |x|^{-d-\alpha} dx,
\end{align}
where $O_\xi$ is an orthonormal matrix such that $O_{\xi}^T \xi = |\xi| e_1$ and $e_1=(1,0,\ldots,0)$. 
Since there is no regularity condition  on the coefficient $a(t,x)$, $\Psi_X(t,\xi)$ is not smooth with respect to $\xi$ in general. 
On the other hand, 
\begin{align*}
\Psi_{X^1}(t,\xi)
&:=\int_{\fR^d} \left( e^{i\xi \cdot x} -1 -i\xi \cdot x 1_{|x| <1} \right)   \mu^1_t(dx)  \\
& =- N(d,\alpha,c) |\xi|^{\alpha}
\end{align*}
and obviously $\Psi_{X^1}(t,\xi)$ satisfies (\ref{main as 3}) and (\ref{main as 4}).
Since for each $t>0$, $\mu$, $\mu^1$, $\mu^2$ are L\'evy measures, there exist additive processes $X_t$, $X^1_t$ and $X^2_t$ such that
\begin{align*}
\bE e^{i \xi \cdot X_t} 
= \exp \left( \int_0^t \Psi_{X}(s,\xi) ds \right),
\end{align*}
%\begin{align*}
%\bE e^{i \xi \cdot X^1_t} 
%= \exp \left( \int_0^t \Phi_{X^1}(s,\xi) ds \right),
%\end{align*}
and
\begin{align*}
\bE e^{i \xi \cdot X^1_t} 
= \exp \left( \int_0^t \Psi_{X^1}(s,\xi) ds \right),
\quad \quad \bE e^{i \xi \cdot X^2_t} 
= \exp \left( \int_0^t \Psi_{X^2}(s,\xi) ds \right)
\end{align*}
due to \cite[Theorem 9.8(ii) and Theorem 11.5]{sato1999levy} again.
We may assume that $X^1$ and $X^2$ are independent.  
One can easily check that $\Psi_{X^2}$ satisfies the assumptions in Example \ref{x2 example}.
Therefore our assumptions hold for the symbol $\Psi_{X}(t,\xi)$  which is defined in (\ref{n smooth s}) and not smooth.
\end{example}

\begin{remark}
%As shown in Example \ref{new exam}, we believe that our results can cover lots of interesting singular symbols 
%even though condition  (\ref{main as 4}) seems to be very strong.
%However, thanks to the referee we found out that our result still does not cover lots of interesting singular symbols.
We acknowledge that there are still some singular symbols which our result cannot cover and were already considered  previously by other authors.  Here, we give an interesting simple example  not covered in this paper.

Let $\alpha \in (0,1)$ and $a(t)$ be a    function  on $[0,\infty)$  such that
$$
0< c \leq a(t) \leq c^{-1} \qquad  \forall t \in [0,\infty). 
$$
For each $t>0$, define
\begin{align*}
&\mu_t(dx)\\
&:=
a(t)  \left(
  |x_1|^{-1-\alpha}dx_1  \cdot\epsilon_0(dx_2, \ldots, dx_d) 
+\cdots    +   |x_d|^{-1-\alpha}d x_d \cdot  \epsilon_0(dx_1, \ldots, dx_{d-1})
\right).
\end{align*}
Here $\epsilon_0$ is the dirac measure centered at zero in $\fR^{d-1}$. 
Then
\begin{align*}
\Psi_{X^1}(t,\xi)
&:=\int_{\fR^d} \left( e^{i\xi \cdot x} -1 -i\xi \cdot x 1_{|x| <1} \right)   \mu_t(dx)  \\
& =- N(d,\alpha)a(t) \left( |\xi_1|^\alpha + \cdots + |\xi_d|^\alpha \right).
\end{align*}
This symbol does not satisfy (\ref{main as 4}).  This example was handled by X. Zhang (see \cite[Remark 2.7]{zhang2013lp})   if $a(t)$ is independent of $t$.  But  the case when   $a(t)$ depends on $t$ is not covered in  literature as far as we know.  
%we do not know yet that a maximal $L_p$-regularity theory still holds. 
%We believe that if the symbol$\Psi(t,\xi)$ contains singularities aside from zero, then it seems to be very difficult to show that the kernel 
%$$
%K(s,t,x,y) := \cF^{-1} \left[  \psi(\xi)  \exp \left( \int_s^t \Psi_X(r,\xi)dr  \right)\right] (x-y)
%$$
%satisfies H\"ormander's condition  by using analytic methods. 
\end{remark}

\subsection{$L_p$-theory for diffusion equations in $\phi$-potential spaces}

In this subsection we present our main result,  the unique solvability of equation (\ref{main eqn}) in $\phi$-potential space.

\begin{defn}[$\phi$-potential space]
						\label{psi sp}
For $\zeta \in C_0^\infty(\fR^d)$, define
$$
\phi(\Delta)\zeta(x)
:=-\phi(-\Delta)\zeta(x)
:=\cF^{-1}\left[\phi(|\xi|^2) \cF[\zeta](\xi) \right](x).
$$
By $H^\phi_{p}$, we denote the space of functions $u \in L_p$ so that
there exists a sequence of $u_n \in C_c^\infty(\fR^d)$ such that
$$
\|u_n - u\|_{L_p} \to 0
$$
and
$$
\|\phi(\Delta)u_n - \phi(\Delta)u_m\|_{L_p} \to 0
$$
as $n,m \to \infty$.  
We call this sequence $u_n$ {\bf a defining sequence} of $u$.
For $u \in H^\phi_p$, we define
$$
\phi(\Delta)u := \lim_{n \to \infty} \phi(\Delta)u_n,
$$
where $u_n$ is a defining sequence of $u$ and the limit is understood in $L_p$-sense. 
\end{defn}

\begin{lemma}

(i) $\phi(\Delta)u$ is well defined for any $u \in H_p^\phi$,  that is,
it is independent of  the choice of defining sequences.

(ii) $H^\phi_{p}$ is a Banach space equipped with the norm
$$
\|u\|_{H^\phi_p} := \|u\|_{L_p} + \|\phi(\Delta) u\|_{L_p}.
$$

(iii) Suppose that Assumption \ref{as} holds. Then 
%$\cA(t) \zeta$
$$
\cA(t)\zeta(x) := \lim_{h \downarrow 0}\frac{\bE\left[\zeta(x+X_{t+h}-X_t)-\zeta(x)\right]}{h} $$
 is well defined for any $\zeta \in C^{\infty}_0(\fR^d)$. 
Moreover there exists an adjoint operator $\cA^*(t)$ so that
\begin{align}
						\label{925 1}
\int_{\fR^d}  \eta(x) \cA(t) \zeta(x)  dx 
=\int_{\fR^d}  \zeta(x) \cA^\ast(t) \eta(x) dx 
\end{align}
for all $\zeta, \eta \in C^{\infty}_0(\fR^d)$.
\end{lemma}
\begin{proof}
First we prove (i). Let $u_n$ and $v_n$ be defining sequences of $u \in H_p^\phi$, respectively.
Then by the Plancherel theorem and Definition \ref{psi sp},
\begin{align*}
\int_{\fR^d} \phi(\Delta) u(x)  \zeta(x) dx
=\lim_{n \to \infty} \int_{\fR^d} \phi(\Delta) u_n(x)  \zeta(x) dx
=\lim_{n \to \infty} \int_{\fR^d}  u_n(x)  \phi(\Delta) \zeta(x) dx \\
=\int_{\fR^d}  u(x)  \phi(\Delta) \zeta(x) dx 
=\lim_{n \to \infty} \int_{\fR^d}  v_n(x)  \phi(\Delta) \zeta(x) dx
=\lim_{n \to \infty} \int_{\fR^d} \phi(\Delta) v_n(x)  \zeta(x) dx
\end{align*}
for all $\zeta \in C_0^\infty(\fR^d)$. Thus $\phi(\Delta)$ is well-defined.

(ii) is obvious due to the property of $L_p$-spaces.

Finally, we prove (iii).
Recall $\Phi_X(t,t,\xi)=0$.
Then due to Assumption \ref{as}(i),
\begin{align*}
\cA(t)\zeta(x) 
&= \lim_{h \downarrow 0}\frac{\bE\left[\zeta(x+X_{t+h}-X_t)-\zeta(x)\right]}{h} \\
&= \lim_{h \downarrow 0}\frac{ \bE \left[\cF^{-1}\left[ \left(e^{i\xi \cdot (X_{t+h} -X_t)}-1 \right)\cF[\zeta](\xi)\right](x) \right]}{h} \\
&= \lim_{h \downarrow 0}\cF^{-1}\left[  \frac{ \exp(\Phi_X(t,t+h,\xi))-\exp(\Phi_X(t,t,\xi))}{h}\cF[\zeta](\xi)\right](x)  \\
&= \cF^{-1}\left[ \Psi_X(t,\xi)\cF[\zeta](\xi)\right](x).
\end{align*}
Thus $\cA(t)$ is well-defined on $C_0^\infty(\fR^d)$ as a pseudo-differential operator since
$$
\Psi_X(t,\cdot)\cF[\zeta](\cdot) \in L_p(\fR^d) \qquad \forall p \in [1,\infty].
$$
Next, define 
$$
\cA^*(t)\zeta(x)=\cF^{-1}\left[ \overline{\Psi_X(t,\xi)}\cF[\zeta](\xi)\right](x).
$$
Then by the Plancherel theorem, (\ref{925 1}) holds. The lemma is proved. 
\end{proof}

\begin{defn}[Definition of  solutions]
						\label{def sol}
For a given $f \in L_p\left( [0,T] ; L_p\right)$,  we say that $u \in L_p\left( [0,T] ; L_{p}\right)$ is a solution to 
%equation (\ref{main eqn})
\begin{align}
								\label{main eqn}
\frac{\partial u}{\partial t}(t,x) = \cA(t)u(t,x) +f(t,x), \quad u(0,\cdot)=0, \quad (t,x) \in (0,T) \times \fR^d,
\end{align}
 %iff
 if 
there exists a sequence $u_n \in C^\infty_c\left( (0,T) \times \fR^d \right)$ such that
$$
\frac{\partial u_n}{\partial t}- \cA(t)u_n \to f \quad \text{in} \quad L_p\left( [0,T] ; L_p\right)
$$
and
$$
u_n \to u \quad \text{in} \quad L_p\left( [0,T] ; H^\phi_{p}\right)
$$
as $n \to \infty$.
\end{defn}

\begin{remark}
            \label{remark defi}
If $u$ is a solution in the sense of Definition \ref{def sol} then it becomes a solution in  the usual weak-sense. 
Indeed, by the Plancherel theorem 
\begin{align*}
&-\int_{(0,T) \times \fR^d} u(t,x) \zeta_t(t,x) dt dx-\int_{(0,T) \times \fR^d} u(t,x) \cA^\ast(t)\zeta(t,x) dt dx \\
&=\lim_{n \to \infty}\int_{(0,T) \times \fR^d} \frac{\partial u_n}{\partial t}(t,x) \zeta(t,x) dt dx
-\lim_{n \to \infty} \int_{(0,T) \times \fR^d} \cA(t)u_n(t,x) \zeta(t,x) dt dx \\
&=\int_{(0,T) \times \fR^d} f(t,x) \zeta(t,x) dt dx \qquad \forall \zeta \in C_c^\infty((0,T) \times \fR^d).
\end{align*}
\end{remark}

Here is the main result of this section.

\begin{thm}
					\label{main thm}
				Suppose that	Assumptions \ref{as 2} and  \ref{as} hold. Then
for any $f \in L_p( [0,T] ; L_p)$,  there exists a unique solution $u \in  L_p\left( [0,T] ; H^\phi_{p}\right)$ to equation (\ref{main eqn}). 
Furthermore, for this solution $u$, we have
\begin{align}
						\label{main est 1}
\| u\|_{L_p\left( [0,T] ; H^\phi_{p}\right)} \leq C_1(d,p,\delta_k,N_j,T) \|f\|_{L_p\left( [0,T] ; L_p\right)}
\end{align}
and
\begin{align}
						\label{main est 2}
\| \phi(\Delta)u\|_{L_p\left( [0,T] ; L_p\right)} \leq C_2(d,p,\delta_k,N_j) \|f\|_{L_p\left( [0,T] ; L_p\right)},
\end{align}
where $\delta_k$ and $N_j$ ($k=1,2,3$ and $j=1,2,3,4$) are constants in Assumption \ref{as}.
\end{thm}
The proof of this theorem will be given in Section \ref{pf main thm}.

\mysection{$L_p$-boundedness of singular integral operators }

In this section we introduce a version of Fefferman-Stein theorem  and Hardy-Littlewood maximal theorem. We also 
prove an $L_p$-boundedness of singular integral operators related to certain pseudo-differential operators.

Let $U=\fR^d$ or $U=\fR^d_+$. For a function $\varphi : (0,\infty)\to (0,\infty)$, 
by $\bQ_\varphi$ we denote the collection of all cubes
$$
Q^\varphi_c(t,x)
=(t,t+\varphi(c)] \times B_c(x),
$$
where $(t,x) \in [0,\infty) \times U$, $c >0$, and $B_c(x) = \{ y \in U : |x-y| < c\}$.

%We introduce a Fefferman-Stein sharp function related to $\varphi$.
\begin{defn}
For locally integrable functions $f$,  denote
$$
f_{Q_c^\varphi(t_0,x_0)}(x)
:=   \aint_{Q_c^\varphi(t_0,x_0)} f(s,y) dsdy:=\frac{1}{|Q_c^\varphi(t_0,x_0)|}  \int_{Q_c^\varphi(t_0,x_0)} f(s,y) dsdy.
$$
The $\varphi$-type sharp function $f_{\varphi}^{\sharp}(t,x)$  and $\bM_{\varphi}f(t,x)$  are defined as  
\begin{align*}
f_{\varphi}^{\sharp}(t,x)
&:= \sup  \aint_{Q_c^\varphi(t_0,x_0)} |f(s,y)-f_{Q_c^\varphi(t_0,x_0)}|dsdy 
\end{align*}
and 
$$
\bM_{\varphi}g(t,x) := \sup  \aint_{Q_c^{\varphi}(t_0,x_0)} |g(s,y)|dsdy,
$$
where $(t,x) \in {(0,\infty)\times U}$,
and sup is taken over all $Q_c^\varphi(t_0,x_0) \in \bQ_\varphi$ containing $(t,x)$.

\end{defn}

\begin{assumption}
					\label{as v}
(i) $\varphi: (0,\infty)\to (0,\infty)$ is a  nondecreasing function  so that
\begin{align}
					\label{varphi first}
\lim_{r\downarrow 0}\varphi(r) =0,\quad  \lim_{r\uparrow \infty}\varphi(r)=\infty,
\end{align}
and
\begin{align}
					\label{tilde c}
\tilde c:=\sup_{r>0} \frac{\varphi( 2 r)}{\varphi(r)} < \infty.
\end{align}

(ii) There exists a constant  $\lambda_0>1 $ satisfying
\begin{align}
						\label{eq 10121}
\varphi(\lambda_0 r) \geq 2 \varphi(r) \quad \quad \forall r>0.
\end{align}
\end{assumption}

Assumption \ref{as v}(i) is sufficient to prove a $\varphi$-type Fefferman-Stein theorem, and condition (\ref{eq 10121}) is additionally needed for
 $\varphi$-type Hardy-Littlewood maximal theorem.

\begin{thm}[$\varphi$-type Fefferman-Stein Theorem]
                \label{fs thm}
Let $p \in (1, \infty)$ and suppose 
(\ref{varphi first}) and (\ref{tilde c}) hold. Then there exists a constant $N$ such that
$$
\|f\|_{L_p({{(0,\infty)\times U}})} \leq N(d,p,\tilde c) \|f_{\varphi}^{\sharp} \|_{L_p({{(0,\infty)\times U}})} \quad \quad \forall f \in L_p({{(0,\infty)\times U}}).
$$
\end{thm}
The proof of Theorem \ref{fs thm} will be given in Section \ref{hl fs pf}.

\begin{thm}[$\varphi$-type Hardy-Littlewood Theorem]
 \label{HL}
 Let $p \in (1, \infty)$ and suppose Assumption \ref{as v} holds. 
 Then for some constant $N=N(p,\varphi)>0$,
\begin{align}
						\label{hl thm}
\| \bM_{\varphi}g \|_{L_p((0,\infty)\times U)}
\leq N\|g \|_{L_p((0,\infty)\times U)} \qquad \forall  g \in L_p((0,\infty)\times U).
\end{align}
\end{thm}

\begin{proof}
One can easily  check that by (\ref{eq 10121})

\noindent
$\bullet$ \,$\exists$ $ \tilde{N}_1>0$ s.t. ${Q_c^{\varphi}}(t,x)\cap Q_c^{\varphi}(s,y)\neq \emptyset$ implies
${Q_c^{\varphi}}(s,y)\subset Q^{\varphi}_{\tilde N_1c}(t,x)$ ;\\
$\bullet$ \,$\exists \, \tilde{N}_2>0$ s.t.  $|Q^{\varphi}_{\tilde N_1c}(t,x)|\leq \tilde N_2 |{Q_c^{\varphi}}(t,x)|$ for all  $(t,x) \in (0,\infty) \times U$ and $c>0$;\\
$\bullet$\,  $\cap_{c>0}\overline{{Q_c^{\varphi}}(t,x)}=(t,x)$ \\
$\bullet$\, for each open set $\cO$ and $c>0$, the function $(t,x)\to
|{Q_c^{\varphi}}(t,x)\cap \cO|$ is continuous.

Hence the theorem follows from the classical Hardy-Littlewood maximal theorem (see \cite[Theorem 1.1]{Stein1993}). 
Actually, \cite[Theorem 1.1]{Stein1993}  is proved on $\fR^{d+1}$. 
The key idea of the proof of \cite[Theorem 1.1]{Stein1993} is Vitali's covering lemma, which holds for arbitrary measurable subset of $\fR^{d+1}$, and  by following the proof,
one can easily check that the Hardy-Littlewood Theorem holds also on $(0,\infty) \times U$.
\end{proof}

Let  $K(s,t,y,x)$ be a  measurable function defined on $(0,\infty)\times (0,\infty) \times U \times U$
so that $K(s,t,y,x)=0$ if $s \geq t$.
For a locally integrable function $f$ on $(0,\infty) \times U$,
denote 
\begin{align*}
\cT f(t,x)
&= \int_0^\infty \int_{U} K(s,t,y,x)f(s,y)dsdy  \\
&= \int_0^t \int_{U} K(s,t,y,x)f(s,y)dsdy  \\
&= \lim_{\varepsilon \downarrow 0}\int_0^{t-\varepsilon} \int_{U} K(s,t,y,x)f(s,y)dsdy \\
&= \lim_{\varepsilon \downarrow 0}\cT_\varepsilon f(t,x),
\end{align*}
where the sense of limit is specified in the following assumption.

\begin{assumption}      \label{bmo as}
For any $f \in L_2((0,\infty) \times U)$,
$\cT_\varepsilon f \to \cT f$ in $L_2((0,\infty) \times U)$ as $\varepsilon \to 0$.
Moreover, the operator $f \mapsto \cT f$ is bounded on $L_2((0,\infty) \times U)$, i.e.
there exists a constant $N_5$ so that
\begin{align*}
\|\cT f\|_{L_2({{(0,\infty)\times U}})} 
\leq N_5\|f\|_{L_2({{(0,\infty)\times U}})} \quad \quad \forall f \in L_2({{(0,\infty)\times U}}).
\end{align*}
\end{assumption}

%Recall that $\varphi$ is a function satisfying Assumption \ref{as v}.
The function $\varphi$ in the next assumption is the one in Assumption \ref{as v}.
\begin{assumption}[$\varphi$-type H\"ormander's condition]      
										\label{bmo as 2}
(i) There exists a function $\tilde{\varphi}:(0,\infty)\to (0,\infty)$ and constants $c_1, c_2>0$ such that  
\begin{equation}
      \label{hor 1}
    r\leq c_1 \varphi(\tilde{\varphi}(r)), \quad \tilde{\varphi}(\varphi(r))\leq c_2 r, \quad \quad \forall r>0.
\end{equation}

(ii) There exist   constants $c_0, N_6>0$  so that
for all $(t,x),(s,y) \in {(0,\infty)\times U}$,
\begin{align}            
					\label{bmo}
 \int_{ \tilde \varphi (|t-r|) +|x-z| \geq c_0( \tilde \varphi (|t-s|) + |x-y|)} | K(r,t,z,x) - K(r,s, z,y)| ~drdz 
\leq  N_6.
\end{align}

\end{assumption}

\begin{remark}
The simplest example of $\tilde \varphi$ above is the generalized inverse of $\varphi$.
Define 
$$
\varphi^{-1}(r) := \inf\{ s: \varphi(s) \geq r\}.
$$
Then obviously $\varphi^{-1}$ is nondecreasing and 
$\varphi^{-1}(  \varphi(r) ) \leq  r$. Also due to 
 (\ref{varphi first}), $0<\varphi^{-1}(r) <\infty$.  
If $\varphi$ is strictly increasing, then
$\varphi^{-1}$ is continuous and 
$\varphi^{-1}(  \varphi(r) ) =  r$.   If, in addition,  $\varphi$ is continuous  then we also have 
$\varphi( \varphi^{-1}(r)) =r$.
Therefore if $\varphi$ is strictly increasing and continuous 
then  we can take $\tilde \varphi(r) = \varphi^{-1}(r)$.
In general, even if   $\varphi$ is neither strictly increasing nor continuous,  due to (\ref{tilde c}) and (\ref{eq 10121}) one can find a constant $\delta>0$ so that
\begin{align*}
\delta^{-1} t \leq \varphi^{-1}(\varphi(t) ) \leq  t, \quad  \delta^{-1} t \leq \varphi(  \varphi^{-1}(t) ) \leq  \delta t
\quad \forall t.
\end{align*}
Therefore  one can still take $\tilde \varphi(r) = \varphi^{-1}(r)$.
\end{remark}

\begin{thm}
				\label{BMO theorem}
Let $p>2$ and suppose that Assumptions \ref{as v}, \ref{bmo as}, and \ref{bmo as 2} hold. 
Then for any $f\in L_2((0,\infty) \times U) \cap L_\infty((0,\infty) \times U)$, 
\begin{enumerate}[(i)]
\item 
\begin{equation}
                     \label{bmo part1}					
\left\|(\cT f)_\varphi^\sharp \right\|_{L_{\infty}( {(0,\infty)\times U} )} \leq N \|f\|_{L_\infty({(0,\infty)\times U})},
\end{equation}
\item 
\begin{equation}
                     \label{bmo part2}
\left\|\cT f\right\|_{L_p( {(0,\infty)\times U} )} \leq N \|f\|_{L_p( {(0,\infty)\times U})},
\end{equation}
\end{enumerate}
where the constant $N$ is independent of $f$. 
\end{thm}

The proof of Theorem \ref{BMO theorem} is based on the following result.

\begin{lemma}       \label{osc lemma}
Suppose that Assumptions \ref{as v}, \ref{bmo as} and \ref{bmo as 2} hold.
Then for any $f\in  L_2({(0,\infty)\times U}) \cap L_\infty({(0,\infty)\times U})$ and $Q^\varphi_c(t_0,x_0) \in \bQ_\varphi$,
\begin{align}
                    \label{mean osc}
\int_{Q^\varphi_c(t_0,x_0)} \int_{Q^\varphi_c(t_0,x_0)} |  \cT f (t,x)- \cT f (s,y)|~dtds dxdy  
\leq   N|Q^\varphi_{c}(t_0,x_0)|^2 \cdot  \sup_{(0,\infty)\times  U} | f|,
\end{align}
where $N$  depends only on $d, \tilde{c}$, $N_5$, and $N_6$. 
\end{lemma}

\begin{proof}
Decompose $f$ into
$f = f \cdot 1_{Q^\varphi_{\delta c}}+  (f-f \cdot 1_{Q^\varphi_{\delta c}}) =: f_1 +f_2$,
where $\delta$ will be specified later.
Then obviously, $f_1$ has a support in $\overline{Q^\varphi_{\delta c}}$ and $f_2$ has a support in the closure of the complement of $Q^\varphi_{\delta c}$. 
First we estimate $\cT f_1$. 
By H\"older's inequality and Assumption \ref{bmo as},
\begin{align}
						\notag
&\int_{Q^\varphi_c(t_0,x_0)} \int_{Q^\varphi_c(t_0,x_0)} |  \cT f_1 (t,x)- \cT f_1 (s,y)|~dtds dxdy \\
						\notag
&\leq 2|Q^\varphi_c(t_0,x_0)|\int_{Q^\varphi_c(t_0,x_0)}  |  \cT f_1 (t,x)|~dtdx \\
						\notag
&\leq 2|Q^\varphi_c(t_0,x_0)|^{3/2} \left(\int_{Q^\varphi_c(t_0,x_0)} |  \cT f_1 (t,x)|^2~dtdx \right)^{1/2}\\
						\notag
&\leq 2N_5|Q^\varphi_c(t_0,x_0)|^{3/2} \left(\int_{Q^\varphi_{\delta c}(t_0,x_0)} |  f_1 (t,x)|^2~dtdx \right)^{1/2}\\
						\label{514}
&\leq N(d,N_5,\delta, \tilde c)|Q^\varphi_c(t_0,x_0)|^{2} \sup_{(0,\infty)\times U}|f|,
\end{align}
where in the last inequality we use the fact that there exists a $n \in \bN$ depending only on $\delta$ so that $
2^{n-1} \leq \delta \leq 2^{n}$
and thus
$\varphi(\delta c) \leq \varphi(2^n c) \leq (\tilde c)^n \varphi (c)$.

Next we estimate $\cT f_2$. Recall
\begin{align*}
\cT f_2(t,x) - \cT f_2(s,y)
= \int_0^\infty \int_{U} (K(r,t,z,x) - K(r,s,z,y)) ( f_2(r,z)) drdz
\end{align*}
and
$f_2(r,z)=0$ if $(r,z) \in Q^\varphi_{\delta c}(t_0,x_0)$.
Note that if 
$$
(t,x),(s,y) \in Q^\varphi_c(t_0,x_0)=(t_0,t_0+\varphi(c)] \times B_c(x_0)
$$
and
\begin{align*}
 \tilde \varphi (|t-r|) +|x-z| < c_0( \tilde \varphi (|t-s|) + |x-y|),
\end{align*}
then by (\ref{hor 1}) and (\ref{tilde c}),
\begin{align*}
|t_0 -r| 
&\leq |t_0 -t| + |t-r| \leq \varphi(c) +  c_1\varphi( \tilde \varphi (|t -r|)) \\
&\leq \varphi(c) +  c_1\varphi\left( c_0 \tilde \varphi (|t -s|) + c_0|x-y|\right) \\
&\leq \varphi(c) +  c_1\varphi( c_0 \tilde \varphi (\varphi(c)) + 2c_0c) \\
&\leq \varphi(c) +  c_1\varphi( c_0c_2c + 2c_0c) \\
&\leq \bar N (c_0,c_1,c_2, \tilde c) \varphi(c)
\end{align*}
and
\begin{align*}
|x_0 -z| 
&\leq |x_0 -x| + |x-z| \\
&\leq c  +  c_0 \tilde \varphi (|t -s|) + c_0|x-y| \\
&\leq c +   c_0 \tilde \varphi (\varphi(c)) +c_0|x-y|  \\
&\leq c +   c_0 c_2 c + 2c_0c \\
&\leq \tilde N(c_0,c_2) c
\end{align*}
Thus taking $\delta > \bar N+ \tilde N$, we have $(r,z) \in Q^\varphi_{\delta c}(t_0,x_0)$ and 
$$
(K(r,t,z,x) - K(r,s,z,y)) ( f_2(r,z))=0.
$$
Therefore by (\ref{bmo}),
\begin{align*}
&\left|\int_0^\infty \int_{U} (K(r,t,z,x) - K(r,s,z,y))   f_2(r,z) drdz \right| \\
& \leq 
\int_{ \tilde \varphi (|t-r|) +|x-z|  \geq  c_0( \tilde \varphi (|t-s|) + |x-y|)} |K(r,t,z,x) - K(r,s,z,y) |  drdz 
\sup_{(0,\infty) \times U} |f_2| \\
& \leq  N_6 \sup_{(0,\infty) \times U} |f_2| \leq N_6 \sup_{(0,\infty) \times U} |f|,
\end{align*}
which certainly implies 
\begin{align}
								\notag
&\int_{Q^\varphi_c(t_0,x_0)} \int_{Q^\varphi_c(t_0,x_0)} |  \cT f_2 (t,x)- \cT f_2 (s,y)|~dtds dxdy  \\
							\label{514 2}
&\quad \leq N_6 |Q^\varphi_{c}(t_0,x_0)|^2 \cdot  \sup_{(0,\infty)\times U} | f|.
\end{align}
Combining (\ref{514}) and (\ref{514 2}), we have (\ref{mean osc}). The lemma is proved. 
\end{proof}

\vspace{2mm}
\noindent
{\bf Proof of Theorem \ref{BMO theorem}}    
\vspace{2mm}

 By Lemma \ref{osc lemma},
\begin{align*}
\|(\cT f)_{\varphi}^\sharp\|_{L_{\infty}({(0,\infty)\times U})}\leq N \|f\|_{L_{\infty}({(0,\infty)\times U})}.
\end{align*}
Thus it is enough to prove    (\ref{bmo part2}).

Obviously, $(\cT f)_{\varphi}^\sharp\leq 2\bM_{\varphi}(\cT f)$.  Thus by Assumption \ref{bmo as} and Theorem \ref{HL},
$$
\|( \cT f)_{\varphi}^\sharp\|_{L_2({(0,\infty)\times U})} \leq N\|f\|_{L_2({(0,\infty)\times U})}.
$$

Note that the  map  $f \to (\cT f)^\sharp$ is subadditive since $\cT$ is a linear operator.
Hence by Marcinkiewicz's interpolation theorem,  for any $p \in (2, \infty)$ there exists a constant $N$ such that
for all  $f\in L_2({(0,\infty)\times U}) \cap  L_\infty({(0,\infty)\times U})$,
$$
\|( \cT f)_{\varphi}^\sharp\|_{L_p({(0,\infty)\times U})} \leq N \|f\|_{L_p({(0,\infty)\times U})}.
$$
Therefore by Theorem \ref{fs thm}, (\ref{bmo part2}) is proved. \qed

\mysection{PDE with pseudo-differential operators}

In this section we study PDEs with pseudo-differential operators.
The result of this section is a  generalization of   Theorem  \ref{main thm}  if  $X^2=0$.

Let $\Psi$ be a complex-valued function defined for $t>0$ and $\xi\in \fR^d$. Consider the equation
\begin{align}
					\label{apl eqn}
u_t= \Psi(t, i D )u + f, \qquad u(0, x)=0,
\end{align}
where 
\begin{align*}
\Psi(t,iD)u(t,x)
:=\cF^{-1} \left[ \Psi(t, \xi) \cF[ u(t, \cdot)](\xi) \right](x).
\end{align*}
Then formally the solution $u$ to equation (\ref{apl eqn}) is given by
\begin{align}
							\label{u rep}
u(t,x)= \cF^{-1}\left[\int_0^t  
\exp \left(\int_s^t \Psi(r,\xi)dr \right) \cF[f(s,\cdot)](\xi)
  ds \right](x)
\end{align}

Recall that $d_0= \left\lfloor \frac{d}{2} \right\rfloor+1$.

\begin{assumption}
							\label{psi as 1}
(i) There exists a nondecreasing function $\psi : (0,\infty) \mapsto (0,\infty)$ and positive constants 
 $\delta_5 \geq \delta_4$, $N_7$, and $N_8$ so that for any $ \lambda_2 \geq \lambda_1>0$
\begin{align}
						\label{psi as}
N_7 \left(\frac{\lambda_2}{\lambda_1}\right)^{\delta_4} \leq  \frac{\psi(\lambda_2)}{\psi(\lambda_1)}  
\leq N_8\left(\frac{\lambda_2}{\lambda_1} \right)^{\delta_5}.
\end{align}

(ii) There exist positive constants $\delta_6$ and $N_9$ so that
\begin{align}
						\label{Psi as 1}
      \Re [\Psi(t,\xi)]  \leq  -\delta_6 \psi(|\xi|^2),\quad
\end{align}
\begin{align}
						\label{Psi as 2}
 |D_\xi^\alpha \Psi(t,\xi)| \leq  N_9|\psi(|\xi|^2)| |\xi|^{-|\alpha|}
\end{align}
for all $(t, \xi) \in (0,\infty) \times \fR^d$ and multi-index $|\alpha| \leq  d_0$.

(iii) $\psi(\lambda)$ is $d_0$-times continuously differentiable and
 there exists a constant $N_{10}$ so that for all $\lambda \in (0,\infty)$ and a natural number $n \leq d_0$,
\begin{align}
						\label{psi dif}
|D^n\psi(\lambda)| \leq N_{10} \lambda^{-n} \psi(\lambda).
\end{align}
\end{assumption}

Denote
$$
\psi^{-1}(t) := \inf\{ s  \geq 0: \psi(s) \geq t\}
$$
Then 
$\psi^{-1}$ is a nondecreasing function from $(0,\infty)$ into $(0,\infty)$ and
there exist positive constants  $\bar N_1$, and $\bar N_2$ so that
for any $\lambda_2 \geq \lambda_1>0$,
\begin{align}
						\label{psi inverse}
\bar N_1 \left(\frac{\lambda_2}{\lambda_1}\right)^{1/\delta_5} 
\leq  \frac{\psi^{-1}(\lambda_2)}{\psi^{-1}(\lambda_1)}  
\leq \bar N_2\left(\frac{\lambda_2}{\lambda_1} \right)^{1/\delta_4}
\end{align}
where $\bar N_1$ and $\bar N_1$ depend only on $\delta_4$, $\delta_5$, $N_7$, and $N_8$.
Furthermore, $\psi( \psi^{-1}(t)) \sim t$ and $\psi^{-1}( \psi(t)) \sim t$, that is 
 for all $t>0$
\begin{align}
						\label{at}
N^{-1} t \leq \psi^{-1}(\psi(t) ) \leq  t, \quad  N^{-1} t \leq \psi(  \psi^{-1}(t) ) \leq  N t,
\end{align}
where $N$ depends only on $\delta_4$, $\delta_5$, $N_7$, and $N_8$.

Here is the main result of this section.

\begin{thm}
					\label{a p e psi}
Let $p \in (1,\infty)$ and suppose Assumption \ref{psi as 1} holds. 
Then for any $f \in L_2((0,\infty)\times \fR^d) \cap L_\infty((0,\infty)\times \fR^d)$
and $u$ defined as in (\ref{u rep}), we have 
\begin{align}
					\label{psi pot}
\| \psi(\Delta)u \|_{L_p((0,\infty)\times \fR^d)}
\leq N \| f \|_{L_p((0,\infty)\times \fR^d)},
\end{align}
 where $N$ depends only on $d$, $p$, $\delta_k$ and $N_j$ ($k=4,5,6$ and $j=7,8,9,10$),  and 
$$ \psi(\Delta)u(t,x) 
:=-\psi(-\Delta)u(t,x)  
:= \cF^{-1} \left[ \psi(|\xi|^2) \cF[u(t,\cdot)](\xi)\right](x).
$$
\end{thm}
We only show this theorem for $p \in [2,\infty)$ due to the duality argument.
To prove this theorem, we apply Theorem \ref{BMO theorem}.
%We first introduce notation to fit in the setting of Theorem \ref{BMO theorem}.
Define
$$
p(s,t,x)= \cF^{-1} \left[ \exp \left(\int_s^t \Psi(r,\xi)dr \right) \right](x),
$$ 
and set
\begin{align*}
K(s,t,y,x)
&:= 1_{0<s<t}(\psi(\Delta))p(s,t,x-y) \\
&:=-1_{0<s<t}(\psi(-\Delta))p(s,t,x-y) \\
&:= 1_{0<s<t}\cF^{-1} \left[ \psi(|\xi|^2)\exp \left(\int_s^t \Psi(r,\xi)dr \right) \right](x-y).
\end{align*}
Note that due to Assumption \ref{psi as 1}(iii), for each $t>s$,
$\psi(|\xi|^2)\exp \left(\int_s^t \Psi(r,\xi)dr \right)$ is integrable with respect to $\xi$
and thus for any $g \in L_2(\fR^d)$,
\begin{align*}
&\cF^{-1} \left[ \psi(|\xi|^2)\exp \left(\int_s^t \Psi(r,\xi)dr \right)\cF g(\xi) \right](x) \\
&=\cF^{-1} \left[ \psi(|\xi|^2)\exp \left(\int_s^t \Psi(r,\xi)dr \right) \right](\cdot) 
\ast g(\cdot)(x) \\
&:=\int_{\fR^d}\cF^{-1} \left[ \psi(|\xi|^2)\exp \left(\int_s^t \Psi(r,\xi)dr \right) \right](x-y) g(y)dy \\
&= \int_{\fR^d}K(s,t,y,x)g(y)dy.
\end{align*}
Therefore (at least formally)
\begin{align*}
\psi(\Delta)u(t,x) 
= \int_0^t \psi(\Delta)p(s,t,x) \ast f(s,x)ds
&=\int_0^t \int_{\fR^d}K(s,t,y,x) f(s,y)dy ds \\
&=\lim_{\varepsilon \downarrow 0 } \int_0^{t-\varepsilon} \int_{\fR^d}K(s,t,y,x) f(s,y)dy ds \\
&=: \lim_{\varepsilon \downarrow 0 } \cT_{\varepsilon} f(t,x) =:\cT f(t,x),
\end{align*}
 where the limit is in $L_2((0,\infty)\times \fR^d)$ (see Lemma \ref{l2 lem}).
 
Set
\begin{align*}
{\varphi}(c):= \psi^{-1}(c^{-1})^{-1/2}, \quad \quad 
{\tilde \varphi}(c):= \psi(c^{-2})^{-1}.
\end{align*}
Then due to (\ref{psi as}) and (\ref{psi inverse}), $\varphi$ and $\tilde \varphi$ satisfy 
$$
\varphi(r) \downarrow 0 \quad \text{as} \quad r \downarrow 0,\quad \varphi(r) \uparrow \infty \quad \text{as} \quad r \uparrow \infty,
$$
$$
\tilde c:=\sup_{r>0} \frac{\varphi( 2 r)}{\varphi(r)} < \infty.
$$
$$r \leq c_1\varphi( \tilde \varphi(r)), \quad  \tilde \varphi(  \varphi(r) ) \leq  c_2 r \qquad \forall r \in (0,\infty).
$$
Thus under this setting, Assumptions \ref{as v} and  (\ref{hor 1}) hold.
Therefore in order to prove (\ref{psi pot}), 
 it suffices to show that Assumption \ref{bmo as} and  (\ref{bmo}) hold.
For this,  we need some preliminaries. 
Denote
$$
a_t := \left(\psi^{-1}(t^{-1}) \right)^{1/2},
$$
$$
\tilde \delta_1 =  \tilde \delta_1(\xi) =
\begin{cases}
&2\delta_4  \quad \text{if} \quad |\xi| \geq 1 \\
&2\delta_5  \quad \text{if} \quad |\xi| < 1,
\end{cases}
$$
and
$$
\tilde \delta_2 = \tilde \delta_2(\xi)=
\begin{cases}
&2\delta_5  \quad \text{if} \quad |\xi| \geq 1 \\
&2\delta_4  \quad \text{if} \quad |\xi| < 1
\end{cases}
$$
\begin{lemma}
					\label{psi lem}
For any $t \in (0,\infty)$ and  $\xi \in \fR^d$,
\begin{align}
						\label{psi at xi}
N^{-1} |\xi|^{\tilde \delta_1}\leq t \psi(|a_t\xi|^2) \leq  N |\xi|^{\tilde \delta_2} 
\end{align}
where $N$ depends only on $\delta_4$, $\delta_5$, $\delta_6$, $N_7$, and $N_8$.
\end{lemma}
\begin{proof}
Due to (\ref{psi as}), there exists a $N$  so that
\begin{align}
						\label{0423 eqn}
N^{-1} |\xi|^{\tilde \delta_1}\leq  \frac{\psi(|a_t\xi|^2)}{ \psi(a^2_t)} \leq  N |\xi|^{\tilde \delta_2} \qquad \forall (t, \xi) \in \fR_+ \times \fR^d.  
\end{align}
Combining (\ref{at}) and (\ref{0423 eqn}), we have
\begin{align*}
\delta^{-1}N^{-1} |\xi|^{\tilde \delta_1}
\leq  t^{-1} \delta^{-1} \frac{t \psi(|a_t\xi|^2)}{ \psi(a^2_t)} 
\leq \psi(a^2_t) \frac{t\psi(|a_t\xi|^2)}{ \psi(a^2_t)} 
\leq  N |\xi|^{\tilde \delta_2}. 
\end{align*}

\end{proof}
\begin{corollary}
					\label{Psi cor}
For any $t \in (0,\infty)$ and  $\xi \in \fR^d$,
\begin{align*}
t \Re [\Psi(r,a_t\xi)] \leq - N |\xi|^{\tilde \delta_1},
\end{align*}
where $N$ depends only on $\delta_4$, $\delta_5$, $\delta_6$, $N_7$, and $N_8$.
\end{corollary}
\begin{proof}
By (\ref{Psi as 1}) and Lemma \ref{psi lem},
\begin{align*}
t \Re [\Psi(r,a_t\xi)] 
& \leq -\delta_6 t \psi( |a_t\xi|^2) \leq -N |\xi|^{\tilde \delta_1}.
\end{align*}
\end{proof}
First, we prove that Assumption \ref{bmo as} holds.

\begin{lemma}
						\label{l2 lem}
There exists a constant $N(d,\delta_6)$ such that
\begin{align}
\|\cT f \|_{L_2((0,\infty)\times \fR^d)}
\leq N\|f \|_{L_2((0,\infty)\times \fR^d)} \qquad \forall  f \in L_2((0,\infty)\times \fR^d).
\end{align}
\end{lemma}
\begin{proof}
By Fubini's theorem, Plancherel's theorem, (\ref{Psi as 1}), and Minkowski's inequality,
\begin{align*}
&\|\cT f \|^2_{L_2((0,\infty)\times \fR^d)} \\
&\leq \int_{\fR^d}  \int_0^\infty \left(\int_0^t \psi(|\xi|^2)e^{\int_s^t \Psi(r,\xi) dr}\cF(f)(s,\xi) ds \right)^2  dt d\xi \\
&\leq \int_{\fR^d}  \int_0^\infty \left(\int_0^t \psi(|\xi|^2) e^{-\delta_3 s \psi(|\xi|^2)}|\cF(f)(t-s,\xi)| ds \right)^2  dt d\xi \\
&\leq \int_{\fR^d} \left(\int_0^\infty   \left(  \int_0^\infty |\cF(f)(t-s,\xi)|^2 dt \right)^{1/2}  \psi(|\xi|^2) e^{-\delta_3 s \psi(|\xi|^2)} ds \right)^2d\xi \\
&\leq N  \int_0^\infty \int_{\fR^d} |f(t,x)|^2 dtdx. 
\end{align*}
The lemma is proved. 
\end{proof}

Next we show that $K$ satisfies (\ref{bmo}).
Denote
\begin{align*}
q_1(s,t,x)=(t-s) \cF^{-1} \left[ \psi(|a_{t-s}\xi|^2) \exp \left(\int_s^t \Psi(r,a_{t-s}\xi)dr \right) \right](x),
\end{align*}
\begin{align*}
q_{2,\ell}(s,t,x)=(t-s) \cF^{-1} \left[  \xi^\ell \psi(|a_{t-s}\xi|^2) \exp \left(\int_s^t \Psi(r,a_{t-s}\xi)dr \right) \right](x),\quad \ell=1, \dots d,
\end{align*}
and
\begin{align*}
&q_3(s,t,x) \\
&= (t-s)^2  \cF^{-1} \left[ \Psi(t, a_{t-s}\xi)   \psi(|a_{t-s}\xi|^2)\exp \left(\int_s^t \Psi(r,a_{t-s}\xi)dr \right) \right](x).
\end{align*}
By the change of variables, 
\begin{align}
						\label{q1 rela}
(t-s)(a_{t-s})^{-d} \psi(\Delta)p(s,t, \cdot)((a_{t-s})^{-1}x)=q_1(s,t,x),
\end{align}
\begin{align}
                    \label{rela p q1}
(t-s)(a_{t-s})^{-d-1}\psi(\Delta)p_{x^j}(s,t, \cdot)( (a_{t-s})^{-1}x)=q_{2,\ell}(s,t,x),
\end{align}
and
\begin{align}
                    \label{rela p q2}
(t-s)^2(a_{t-s})^{-d}\frac{\partial}{\partial t} \psi(\Delta) p(s,t,\cdot)((a_{t-s})^{-1}x) 
=q_3(s,t,x).
\end{align}

Note that 
by (\ref{psi dif}), (\ref{Psi as 1}), (\ref{Psi as 2}), Lemma \ref{psi lem}, and Corollary \ref{Psi cor}, there exists a positive constant $N$ such that for all $\xi \neq 0$, 
\begin{align}
						\notag
& \left|D^\alpha_{\xi}\left(\cF(q_1(s,t,\cdot)(\xi) \right)  \right| \\
						\notag
&=(t-s)\left|D^\alpha_{\xi}\left( \psi(|a_{t-s}\xi|^2) \exp \left(\int_s^t \Psi(r,a_{t-s}\xi)dr \right) \right)  \right| \\
								\label{903 1}
&\leq N\left||\xi|^{ \tilde \delta_2 -|\alpha|} \exp \left(-N^{-1}|\xi|^{\tilde \delta_1} \right)   \right|.
\end{align}

\begin{lemma}
						\label{902 lem 1}
There exists a constant $N=N(d,\delta_k,N_j)$ ($k=4,5,6$ and $j=7,8,9,10$) so that 
for any multi-index $\alpha$ with $|\alpha|\leq d_0$, $0<s<t$, and $\ell=1,\ldots,d$,
\begin{align*}
&\int_{\fR^d} \left|D^\alpha_{\xi}\left(\cF[q_1(s,t,\cdot)](\xi) \right)  \right| d\xi
+\int_{\fR^d} \left|D^\alpha_{\xi}\left( \cF[q_{2,\ell}(s,t,\cdot)](\xi) \right)  \right| d\xi \\
&\quad +\int_{\fR^d} \left|D^\alpha_{\xi}\left( \cF[q_3(s,t,\cdot)](\xi) \right)  \right| d\xi \leq N.
\end{align*}
\end{lemma}
\begin{proof}
The first term 
$$
\int_{\fR^d} \left|D^\alpha_{\xi}\left(\cF[q_1(s,t,\cdot)](\xi) \right)  \right| d\xi
$$
is easily controlled by \eqref{903 1}.
Indeed, since
$$
\left||\xi|^{ \tilde \delta_2 -|\alpha|} \exp \left(-N^{-1}|\xi|^{\tilde \delta_1} \right)   \right|
\leq N\left||\xi|^{ \tilde \delta_2 -d_0} \exp \left(-(2N)^{-1}|\xi|^{\tilde \delta_1} \right)   \right|
$$
and the latter function is integrable with respect to $\xi$, we have
$$
\int_{\fR^d} \left|D^\alpha_{\xi}\left(\cF[q_1(s,t,\cdot)](\xi) \right)  \right| d\xi \leq N.
$$
The other two terms are similarly controlled by the inequalities
$$
\left|D^\alpha_{\xi}\left( \cF[q_{2,\ell}(s,t,\cdot)](\xi) \right)  \right|
\leq  N\left||\xi|^{ \ell+\tilde \delta_2 -d_0} \exp \left(-(2N)^{-1}|\xi|^{\tilde \delta_1} \right)   \right|
$$
and
$$
\left|D^\alpha_{\xi}\left( \cF[q_3(s,t,\cdot)](\xi) \right)  \right|
\leq  N\left||\xi|^{ 2\tilde \delta_2 -d_0} \exp \left(-(2N)^{-1}|\xi|^{\tilde \delta_1} \right)   \right|.
$$
 The lemma is proved.
\end{proof}

Note that for any $f \in L_1(\fR^d)$,
\begin{align*}
\sup_{x \in \fR^d} \left|\cF^{-1}(f)(x)\right|  \leq N(d)\|f\|_{L_1(\fR^d)}.
\end{align*}
Thus by Lemma \ref{902 lem 1} with $\alpha=0$,  there exists a constant 
$N=N(d,\delta_k,N_j)$ ($k=4,5,6$ and $j=7,8,9,10$)
so that for any $t>s$, $x \in \fR^d$, $\ell=1,\ldots,d$,
\begin{align}
						\label{ker bd eq}
\left|q_1(s,t, x)\right|
+\left| q_{2,\ell}(s,t, x)\right|
+| q_3(s,t, x)| \leq N.
\end{align}
\begin{lemma}
						\label{ker bd lem 2}
Let $\varepsilon \in \left[0,2\delta_4 +\frac{d}{2} - (d_0-1)\right)$. 
Then, there exists a constant 
%$N=N(d,\delta_k,N_k)$
$N=N(d,\delta_k,N_j)$ ($k=4,5,6$ and $j=7,8,9,10$)
so that for any multi-index $\alpha$ with $|\alpha|\leq d_0-1$, $0<s<t$, and $\ell=1,\ldots,d$,
\begin{align*}
&\int_{\fR^d} \left||\xi|^{-\varepsilon}D^\alpha_{\xi}\left(\cF[q_1(s,t,\cdot)](\xi) \right)  \right|^2 d\xi
+\int_{\fR^d} \left||\xi|^{-\varepsilon}D^\alpha_{\xi}\left( \cF[q_{2,\ell}(s,t,\cdot)](\xi) \right)  \right|^2 d\xi\\
&\quad+\int_{\fR^d} \left||\xi|^{-\varepsilon}D^\alpha_{\xi}\left( \cF[q_3(s,t,\cdot)](\xi) \right)  \right|^2 d\xi 
\leq N.
\end{align*}
\end{lemma}
\begin{proof}
Because of the similarity, we only show 
$$
\int_{\fR^d} \left||\xi|^{-\varepsilon}D^\alpha_{\xi}\left( \cF[q_1(s,t,\cdot)](\xi) \right)  \right|^2 d\xi \leq N.
$$
Due to (\ref{903 1}),
\begin{align*}
 \left||\xi|^{-\varepsilon}D^\alpha_{\xi}\left( \cF[q_1(s,t,\cdot)](\xi) \right)  \right|^2 
 \leq N\left||\xi|^{ 2\tilde \delta_2 -2|\alpha|-2\varepsilon} \exp \left(-2N^{-1}|\xi|^{\tilde \delta_1} \right)   \right|.
\end{align*}
Therefore, 
$$
\left||\xi|^{-\varepsilon}D^\alpha_{\xi}\left( \cF[q_1(s,t,\cdot)](\xi) \right)  \right|^2 
$$
is integrable with respect to $\xi$ uniformly for $0<s<t$ since
$$
2\tilde \delta_2 -2|\alpha|-2\varepsilon > 4\delta_4 - 2(d_0-1) -2\varepsilon > -d.
$$
The lemma is proved.
\end{proof}

\begin{lemma}
						\label{ker bd lem 3}
There exists a constant $N=N(d,\delta_k,N_j)>0$ ($k=4,5,6$ and $j=7,8,9,10$) so that 
for all $c>0$, multi-index $|\alpha|\leq d_0$, $0<s<t$, and $\ell=1,\ldots,d$,
\begin{align*}
&\int_{|\xi| \geq c} \left|D^\alpha_{\xi}\left(\cF[q_1(s,t,\cdot)](\xi) \right)  \right|^2 d\xi
+\int_{|\xi| \geq c} \left|D^\alpha_{\xi}\left(\cF[q_{2,\ell}(s,t,\cdot)](\xi) \right)  \right|^2 d\xi \\
&\quad +\int_{|\xi| \geq c} \left|D^\alpha_{\xi}\left( \cF[q_3(s,t,\cdot)](\xi) \right)  \right|^2 d\xi 
\leq N(1+c^{4\delta_2-2d_0 +d}).
\end{align*}
\end{lemma}
\begin{proof}
Due to similarity, we only estimate the first term above.

By (\ref{903 1}),
\begin{align*}
 \left|D^\alpha_{\xi}\left( \cF[q_1(s,t,\cdot)](\xi) \right)  \right|^2 
 &\leq N|\xi|^{ 2\tilde \delta_2 -2|\alpha|} \exp \left(-N^{-1}|\xi|^{\tilde \delta_1} \right)\\
 &\leq N|\xi|^{ 4\delta_4 -2d_0} \exp \left(-(2N)^{-1}|\xi|^{\tilde \delta_1} \right) .
\end{align*}
Therefore
\begin{align*}
\int_{|\xi| \geq c} \left|D^\alpha_{\xi}\left(\cF[q_1(s,t,\cdot)](\xi) \right)  \right|^2 d\xi
\leq N(1+c^{ 4\delta_4 -2d_0+d}).
\end{align*}
The lemma is proved.
\end{proof}

\begin{lemma}   \label{ker int fin}
Let $0< \delta < \left(\delta_4 \wedge \frac{1}{2} \right) $.
Then there exists a constant $N=N(d,\delta,\delta_k,N_j)$ ($k=4,5,6$ and $j=7,8,9,10$)
so that for any $0<s<t$ and $\ell=1,\ldots,d$
\begin{align}           \label{lemma 531}
 \int_{\fR^d} \left| |x|^{\frac{d}{2} +\delta } q_1(s,t, x) \right|^2~dx \leq N,
\end{align}
\begin{align}               \label{lemma 532}
 \int_{\fR^d} \left| |x|^{\frac{d}{2}+\delta}  q_{2,\ell}(s,t, x) \right|^2~dx \leq N,
\end{align}
and
\begin{align}               \label{lemma 533}
\int_{\fR^d} \left| |x|^{\frac{d}{2}+\delta}q_3(s,t, x) \right|^2~dx \leq N.
\end{align}
\end{lemma}
\begin{proof}
As before, we only prove (\ref{lemma 531}) since the  proofs of (\ref{lemma 532}) and (\ref{lemma 533}) are similar.

Note that it suffices to show that for each $\ell=1,\ldots,d$,
\begin{align}
								\label{808 1}
 \int_{\fR^d} \left| (ix^{\ell})^{\frac{d}{2} +\delta }q_1(s,t, x) \right|^2~dx \leq N,
\end{align}
where $i$ is the imaginary number, i.e. $i^2=-1$.
By a property of the Fourier inverse transform,
$$
(ix^\ell)^{d_0-1}\cF^{-1}\left[f(\xi)\right](x)
=(-1)^{d_0-1}\cF^{-1}\left[D^{d_0-1}_{\xi^\ell}f(\xi)\right](x).
$$
Hence the left hand side of (\ref{808 1}) is equal to
\begin{align}
					\notag
& \int_{\fR^d} \left| (ix^\ell)^{\frac{d}{2} +\delta -(d_0-1)} \cF^{-1}\left( D_{\xi^\ell}^{d_0-1} \cF\left[q_1(s,t,\cdot) \right](\xi)\right)(x) \right|^2~dx \\
					\label{808 2} 
& \leq \int_{\fR^d} \left| |x|^{\frac{d}{2} +\delta -(d_0-1)} \cF^{-1}\left( D_{\xi^\ell}^{d_0-1}  \cF \left[q_1(s,t,\cdot) \right] (\xi) \right)(x) \right|^2~dx.
\end{align}
Set
$$
\varepsilon:=\varepsilon(\delta)=\frac{d}{2} +\delta -(d_0-1).
$$
Then by the Plancherel theorem, the right hand side of (\ref{808 2})  equals 
\begin{align}
						\label{808 3}
N(d)\int_{\fR^d} \left| (-\Delta)^{\varepsilon/2} \left(D_{\xi^\ell}^{d_0-1} q_1(s,t,\xi) \right) \right|^2~d\xi.
\end{align}
Obviously, $\varepsilon \in \left(0,1 \wedge \left(2\delta_4 + \frac{d}{2} -(d_0-1)\right) \right)$. Using the integral representation of the Fractional Laplacian operator 
$(-\Delta)^{\varepsilon/2}$ we get
\begin{align*}
(-\Delta)^{\varepsilon/2} (D_{\xi^\ell}^{d_0-1} q_1(s,t,\xi))
=N \int_{\fR^d} \frac{ D_{\xi^\ell}^{d_0-1} q_1(s,t,\xi+\eta)-D_{\xi^\ell}^{d_0-1} q_1(s,t,\xi) }{|\eta|^{d+\varepsilon}} d\eta.
\end{align*}
We divide $(-\Delta)^{\varepsilon/2} (D_{\xi^\ell}^{d_0-1} q_1(s,t,\xi))$ into two terms:
\begin{align*}
&N \int_{|\eta| \geq 1} \frac{ D_{\xi^\ell}^{d_0-1} q_1(s,t,\xi+\eta)-D_{\xi^\ell}^{d_0-1} q_1(s,t,\xi) }{|\eta|^{d+\varepsilon}} d\eta \\
&\quad+N \int_{|\eta| < 1} \frac{ D_{\xi^\ell}^{d_0-1} q_1(s,t,\xi+\eta)-D_{\xi^\ell}^{d_0-1} q_1(s,t,\xi) }{|\eta|^{d+\varepsilon}} d\eta
=: \cI_1(s,t,\xi) + \cI_2(s,t,\xi).
\end{align*}
By Minkowski's inequality and Lemma \ref{ker bd lem 2},
\begin{align*}
\left[\int_{\fR^d} \left|\cI_1(s,t,\xi) \right|^2 d\xi \right]^{1/2}
&\leq 2 \left\| D_{\xi^\ell}^{d_0-1} q_1(s,t,\cdot)\right\|_{L_2(\fR^d)}
\int_{|\eta| \geq 1} \frac{1}{|\eta|^{d+\varepsilon}} d\eta \leq N<\infty.
\end{align*}
We split $\cI_2$ into $\cI_{2,1}$, $\cI_{2,2}$, and $\cI_{2,3}$, where
\begin{align*}
\cI_{2,1}(s,t,\xi):=
\int_{|\eta| < 1} 1_{|\eta| < \frac{|\xi|}{2}}\frac{ D_{\xi^\ell}^{d_0-1} q_1(s,t,\xi+\eta)-D_{\xi^\ell}^{d_0-1} q_1(s,t,\xi) }{|\eta|^{d+\varepsilon}} d\eta,
\end{align*}
\begin{align*}
\cI_{2,2}(s,t,\xi):=\int_{|\eta| < 1} 1_{|\eta| \geq \frac{|\xi|}{2}} \frac{ D_{\xi^\ell}^{d_0-1} q_1(s,t,\xi+\eta)}{|\eta|^{d+\varepsilon}} d\eta,
\end{align*}
and
\begin{align*}
\cI_{2,3}(s,t,\xi):=-\int_{|\eta| < 1} 1_{|\eta| \geq \frac{|\xi|}{2}} \frac{D_{\xi^\ell}^{d_0-1} q_1(s,t,\xi) }{|\eta|^{d+\varepsilon}} d\eta.
\end{align*}
By the fundamental theorem of calculus and the Fubini theorem,
\begin{align*}
|\cI_{2,1}(s,t,\xi)| 
\leq 
\int_0^1 \int_{|\eta| < 1}  1_{|\eta| < \frac{|\xi|}{2}}\frac{ \left|\nabla D_{\xi^\ell}^{d_0-1} q_1(s,t,\xi+ \theta \eta) \right|}{|\eta|^{d+\varepsilon-1}}  d\eta d\theta .
\end{align*}
Hence by Minkowski's inequality and Lemma \ref{ker bd lem 3}, 
\begin{align*}
\|\cI_{2,1}(s,t,\cdot)\|_{L_2(\fR^d)}
&\leq 
 \int_{|\eta| < 1} \left(\int_{|\eta| < |\xi|}  \left|\nabla D_{\xi^\ell}^{d_0-1} q_1(s,t,\xi)\right|^2d\xi \right)^{1/2}\frac{1}{|\eta|^{d+\varepsilon-1}} d\eta  \\
&\leq 
N \int_{|\eta| < 1} \frac{1+|\eta|^{2\delta_4-d_0 +\frac{d}{2}}}{|\eta|^{d+\varepsilon-1}} d\eta 
\leq N
\end{align*}
since 
$2\delta_4-d_0 +\frac{d}{2}-d -\varepsilon +1 >-d$.

Note that 
if $|\xi| \geq 2$, then $\cI_{2,2}(s,t,\xi)=\cI_{2,3}(s,t,\xi)=0$ and thus we may assume $|\xi| \leq 2$.
Recalling the range of $\varepsilon$, we have
$$
\varepsilon  + \delta_4 < 2\delta_4 +\frac{d}{2} -(d_0-1).
$$
Hence  by H\"older's inequality and Lemma \ref{ker bd lem 2},
\begin{align*}
&|\cI_{2,2}(s,t,\xi)| \\
&\leq 
\left[\int_{|\eta| < 1} 1_{|\eta| \geq \frac{|\xi|}{2}} \frac{|\xi+\eta|^{2\varepsilon+2\delta_4}}{|\eta|^{2d+2\varepsilon}} d\eta\right]^{1/2}
\left[\int_{\fR^d} \left| |\xi+\eta|^{-\varepsilon -\delta_4 }D_{\xi^\ell}^{d_0-1} q_1(s,t,\xi+\eta) \right|^2 d\eta\right]^{1/2} \\
&\leq N
\left[\int_{|\eta| < 1} 1_{|\eta| \geq \frac{|\xi|}{2}}|\eta|^{-2d+2\delta_4} d\eta\right]^{1/2}
\left[\int_{\fR^d} \left| |\eta|^{-\varepsilon-\delta_4}D_{\xi^\ell}^{d_0-1} q_1(s,t,\eta) \right|^2 d\eta\right]^{1/2} \\
&\leq N \left(1+|\xi|^{-\frac{d}{2} +\delta_4}\right).
\end{align*}
Therefore we have
\begin{align*}
\|\cI_{2,2}(s,t,\cdot)\|^2_{L_2(\fR^d)}
\leq N \int_{|\xi| <2} \left(1+|\xi|^{-d+2\delta_4}\right) d\xi \leq N.
\end{align*}
Finally by Lemma \ref{ker bd lem 2} again,
\begin{align*}
\|\cI_{2,3}(s,t,\cdot)\|^2_{L_2(\fR^d)}
\leq  N \int_{|\xi| \leq 2}\left(1+|\xi|^{-d +2 \delta_4}\right)  d\xi
\leq N.
\end{align*}
Due to (\ref{808 2}) and (\ref{808 3}), combining all estimates for $\cI_1, \cI_{2,1}, \cI_{2,2}, \cI_{2,3}$, we have
(\ref{808 1}). 
The lemma is proved. 
\end{proof}

Let $q(s,t,x)$ be anyone of $q_1, q_{2,\ell}$ and $q_3$. Then by (\ref{ker bd eq}), Lemma \ref{ker int fin}, and H\"older's inequality, 
\begin{align}
        \label{eqn 11.17.1}
       & \sup_{s<t}\|q(s, t, \cdot)\|_{L_1}\nonumber \\
       & \leq \sup_{s>t} \int_{|x|\leq 1} |q(s,t,x)|dx+N\sup_{s<t} \left(\int_{|x|\geq 1}  \left||x|^{d/2+\delta}q(s,t,x)\right|^2 dx\right)^{1/2}<\infty.
\end{align}        
        
Recall
$a_t := \left(\psi^{-1}(t^{-1}) \right)^{1/2}$ and denote
$$
\psi (\Delta) p(r,t,z)
:=\psi (\Delta) \left[ p(r,t,\cdot) \right](z)
= \cF^{-1}\left[\psi(|\xi|^2) \cF \left[p(r,t,\cdot) \right] (\xi) \right](z).
$$
\begin{lemma}           
					\label{main cor}
Let $0< \delta < \left(\delta_4 \wedge \frac{1}{2} \right)$.
Then there exists a constant $N=N(d,\delta,\delta_k,N_j)$ ($k=4,5,6$ and $j=7,8,9,10$)
such that for all $t>s>a>0$, $c>0$,
\begin{align}
					\label{m c 1}
\int_s^t \int_{ |z| \geq c} |\psi (\Delta) p(r,t,z) (z)| ~dz dr 
\leq N \left(a_{t-s}c\right)^{-\delta},
\end{align}
\begin{align}       
					\label{m c 2}
\int_{0}^a \int_{\fR^d} \big|\psi(\Delta)p(r,t, z+h)-\psi(\Delta)p(r,t, z)\big| ~dz  dr 
\leq N |h| a_{t-a},
\end{align}
and
\begin{align}
						\label{m c 3}
\int_{0}^a  \int_{\fR^d} |\psi(\Delta)p(r,t, z)-  \psi(\Delta)p(r,s, z)| ~dz dr
\leq N(t-s)(s -a)^{-1}.
\end{align}
\end{lemma}
\begin{proof}
(i)
By (\ref{q1 rela}), (\ref{lemma 531}), and H\"older's inequality,
\begin{align*}
&\int_{|z| \geq c} |\psi(\Delta)p(r,t, z)| ~dz \\
&=(t-r)^{-1}\int_{|z| \geq a_{t-r}c} |q_1(r,t, z)| ~dz \\
&\leq (t-r)^{-1} 
\left(\int_{|z| \geq a_{t-r}c} |z|^{-d - 2\delta} ~dz \right)^{1/2} 
\left(\int_{|z| \geq a_{t-r}c} \left||z|^{\frac{d}{2}+\delta} q_1(r,t, z)\right|^2 ~dz\right)^{1/2} \\
&\leq N (t-r)^{-1} \left(a_{t-r}c \right)^{- \delta}.
\end{align*}
Therefore by (\ref{at}) and changing the variable $r \to t-(t-s)r$,
\begin{align*}
\int_s^t \int_{|z| \geq c} |\psi(\Delta)p(r,t, z)| ~dz dr
&\leq N \int_s^t(t-r)^{-1} \left(a_{t-r}c \right)^{- \delta} dr \\
&\leq N \int_0^{1}r^{-1} \left(a_{(t-s)r}c \right)^{- \delta} dr.
\end{align*}
Thus by (\ref{psi inverse}),
\begin{align*}
\int_0^{1}r^{-1} \left(a_{(t-s)r}c \right)^{- \delta} dr
& \leq N \left(a_{t-s}c \right)^{- \delta} \int_0^{1}r^{-1} r^{\delta / (2 \delta_5)} dr 
\leq N \left(a_{t-s}c \right)^{- \delta}.
\end{align*}

(ii)
Recall
\begin{align*}  
\psi(\Delta)p_{x^\ell}(r,t, x)= (t-r)^{-1}(a_{t-r})^{d+1} q_{2,\ell}(r,t,a_{t-r}x).
\end{align*}
Using the fundamental theorem of calculus,  Fubini's theorem, and (\ref{eqn 11.17.1}), 
\begin{align*}
&\int_{0}^a \int_{\fR^d} \big| \psi(\Delta)p(t,r, z+h)-\psi(\Delta)p(r,t, z) \big| ~dz  dr  \\
&\leq |h|\int_{0}^a  \int_{\fR^d} \int^1_0 |\nabla \psi(\Delta)p(r,t, z+\theta h)| ~d\theta dz dr  \\
&\leq |h|\int_{0}^a (t-r)^{-1}a_{t-r}\sum_{\ell=1}^d \int_{\fR^d} \Big|q_{2,\ell}(r,t, z) \Big|~dz dr 
\leq N|h|\int_{0}^a (t-r)^{-1}a_{t-r} dr .
\end{align*}
Moreover, by changing the variable $r \to (t-a)r$ and (\ref{psi inverse}),
\begin{align*}
\int_{0}^a (t-r)^{-1}a_{t-r} dr 
\leq \int_{1}^\infty r^{-1}a_{(t-a)r} dr 
&=a_{t-a}\int_{1}^\infty r^{-1}\frac{a_{(t-a)r}}{a_{t-a}} dr  \\
&\leq N a_{t-a}.
\end{align*}
Hence (\ref{m c 2}) is proved. 

(iii)
By the fundamental theorem of calculus and (\ref{rela p q2}),
\begin{align*}
&|\psi(\Delta)p(r,t, z)-  \psi(\Delta)p(r,s, z)| \\
&\leq \int_0^1 |t-s|(\theta t + (1-\theta)s - r)^{-2}(a_{(\theta t + (1-\theta)s-r)})^d \\
&\qquad \times |q_3(r,\theta t + (1-\theta)s, a_{(\theta t + (1-\theta)s-r)}z)|d \theta.
\end{align*}
Therefore, by  (\ref{eqn 11.17.1}),
\begin{align*}
\int_{0}^a \int_{\fR^d} |\psi(\Delta)p(r,t, z)-  \psi(\Delta)p(r,s, z)| ~dz dr
&\leq \int_{0}^a \frac{|t-s|}{\big(\theta t + (1-\theta)s -r \big)^2} dr\\
&\leq |t-s|(s-a)^{-1}.
\end{align*}
The lemma is proved.
\end{proof}
Recall
\begin{align*}
{\varphi}(c) = \psi^{-1}(c^{-1})^{-1/2}=1/a_c
\end{align*}
and observe that by (\ref{psi inverse}), there exists a $\tilde c_0 \geq 1$ so that
\begin{align}
						\label{psi c0}
{\varphi}(t+s) \leq  \tilde c_0 \left({\varphi}(t) + {\varphi}(s) \right)  \qquad \forall s,t  \geq 0.
\end{align}
Denote
$$
A(t,s,r,y,x):= \left\{z \in \fR^d: {\varphi} (|t-r|) +|x-z| \geq 4\tilde c_0( {\varphi} (|t-s|) + |x-y|) \right\}.
$$ 
\begin{corollary}
For all $(t,x),(s,y) \in {(0,\infty)\times \fR^{d}}$,
\begin{align}            
						\label{326 eqn 9}
\int_{0}^\infty \int_{A(t,s,r,y,x)} | 1_{0 < r< t}\psi(\Delta)p(r,t,x-z) - 1_{0<r< s}\psi(\Delta)p(r,s, y-z)| ~dzdr 
\leq N,
\end{align}
where 
$N=N(d,\delta_k,N_j)$ ($k=4,5,6$ and $j=7,8,9,10$).
\end{corollary}

\begin{proof}
Choose a $0< \delta < \left(\delta_4 \wedge \frac{1}{2} \right)$.
Without loss of generality, we assume $t\geq s$. 
Denote
\begin{align*}
\cI(t,s,r,y,x)
=\int_{A(t,s,r,y,x)} | 1_{0<r<t}\psi(\Delta)p(r,t,x-z) - 1_{0<r<s}\psi(\Delta)p(r,s, y-z)| ~dz.
\end{align*}
If $r \geq t$, then $\cI(t,s,r,y,x)=0$. Thus 
\begin{align*}            
\int_{0}^\infty \cI(t,s,r,y,x) dr 
&=\int_{2s-t}^t \cI(t,s,r,y,x) dr +\int_{0}^{2s-t} \cI(t,s,r,y,x) dr \\
&=:\cI_1(t,s,y,x) + \cI_2(t,s,y,x).
\end{align*}
First we estimate $\cI_1(t,s,y,x)$. 
Note that due to (\ref{psi c0}),
\begin{align}
						\label{1126 e 1}
A(t,s,r,y,x) \subset \{ z \in \fR^d : |x-z| \geq  {\varphi} (|t-s|)  \}
\end{align}
if $2s-t<r<t$.
By  (\ref{1126 e 1}) and (\ref{m c 1}),
\begin{align*}
\cI_1(t,s,y,x)
&\leq \int_{2s-t}^t \int_{{\varphi}(|t-r|) +|x-z| \geq 4\tilde c_0({\varphi}(|t-s|) + |x-y|)} | \psi(\Delta)p(r,t,x-z) | ~dz dr \\
&\quad +\int_{2s-t}^t \int_{{\varphi}(|t-r|) +|x-z| \geq 4\tilde c_0({\varphi}(|t-s|) + |x-y|)} | \psi(\Delta)p(r,s,x-z) | ~dz dr \\
&\leq 2\int_{2s-t}^t\int_{|z| \geq  {\varphi} (|t-s|) } | \psi(\Delta)p(r,t,z) | ~dz dr \\
&\leq N \left(a_{t-s} {\varphi} (|t-s|) \right)^{-\delta}\leq N.
\end{align*}
We split $\cI_2$.
Observe
\begin{align*}
\cI_2 
&\leq \cI_{2,1}+ \cI_{2,2}\\
&:= \int_{0}^{2s-t}  \int_{A(t,s,r,y,x)} | 1_{0<r<t}\psi(\Delta)p(r,t,x-z) - 1_{0<r<t}\psi(\Delta)p(r,t, y-z)| ~dzdr  \\
&\quad +\int_{0}^{2s-t}  \int_{A(t,s,r,y,x)} | 1_{0<r<t}\psi(\Delta)p(r,t,y-z) - 1_{0<r<s}\psi(\Delta)p(r,s, y-z)| ~dz dr.
\end{align*}
If $|x-y| \leq {\varphi}  ((t-s))$ then by (\ref{m c 2}),
\begin{align*}
\cI_{2,1} \leq N|x-y| a_{2(t-s)} \leq N.
\end{align*}
On the other hand, if $|x-y| > {\varphi} ((t-s))$, then 
\begin{align}
						\label{1126 e 2}
t-s  \leq \frac{1}{ \psi(|x-y|^{-2})}.
\end{align}
Moreover by (\ref{psi c0}), if  $s-\left( \psi( |x-y|^{-2}) \right)^{-1}<r<t$ and (\ref{1126 e 2}) holds, then 
$$
A(t,s,r,y,x) \subset \{ |x-z| \geq {\varphi}(|t-s|) + |x-y| \}.
$$
Therefore
\begin{align*}
\cI_{2,1}\leq 2\cI_{2,1,1}+\cI_{2,1,2}, 
\end{align*}
where
\begin{align*}
&\cI_{2,1,1}:= \int_{s-\left( \psi( |x-y|^{-2}) \right)^{-1}}^{t} \int_{|z| \geq {\varphi}(|t-s|) + |x-y|} | \psi(\Delta)p(r,t,z)| ~dz dr,
\end{align*}
and 
\begin{align*}
\cI_{2,1,2}:=\int_{0}^{s-\left( \psi( |x-y|^{-2}) \right)^{-1}}  \int_{\fR^d}  1_{0<r<t}\left|\psi(\Delta)p(r,t,x-z) - \psi(\Delta)p(r,t, y-z)\right| ~dz dr.
\end{align*}
Recalling
$$
a_{t} := \left(\psi^{-1}(t^{-1}) \right)^{1/2} = 1/\varphi(t),
$$
we have by (\ref{m c 1}) again 
$$
\cI_{2,1,1} \leq N \left( a_{\left(t-s + \left( \psi( |x-y|^{-2}) \right)^{-1}\right)}\left({\varphi} ((t-s))+|x-y|\right)\right)^{-\delta} \leq N
$$
and by (\ref{m c 2})
\begin{align*}
\cI_{2,1,2} \leq |x-y| a_{\left(t - s +\left( \psi( |x-y|^{-2}) \right)^{-1} \right)} 
\leq N.
\end{align*}
It only remains to estimate $\cI_{2,2}$, which is an easy consequence of (\ref{m c 3}).
Indeed,
\begin{align*}
\cI_{2,2} \leq N(t-s)(t-s)^{-1} \leq N
\end{align*}
since $2s-t < s$. 
The corollary is proved.
\end{proof}

\mysection{Proof of Theorem \ref{main thm}}
										\label{pf main thm}
In this section, $X$ is a stochastic process satisfying Assumptions \ref{as 2} and  \ref{as}.
First we introduce the representation of solutions and related estimates.
\begin{lemma}
Let $f$ be a smooth function on $(0,T) \times \fR^d$ such that for any multi-index $\alpha$ and $\beta$,
\begin{align}
					\label{e 1030 1}
\sup_{t \in (0,T)} \sup_{x \in \fR^d}\left|x^\beta D^\alpha_xf(t,x) \right| < \infty
\end{align}
and  suppose that 
%Assumption \ref{as} holds.
Assumptions \ref{as 2} and  \ref{as} hold.
Define 
$$
u(t,x) := \int_0^t \bE\left[ f(s,x+X_t -X_s) \right]ds.
$$
Then
\begin{align}
								\label{a sol}
u_t(t,x) = \cA(t)u(t,x)+ f(t,x), \quad u(0,x)=0,
\end{align}
for almost every $(t,x) \in (0,T) \times \fR^d$ .
Moreover,
\begin{align}
						\label{a est 1}
\| u\|_{L_p\left( [0,T] ; H^\phi_{p}\right)} \leq N \|f\|_{L_p\left( [0,T] ; L_p\right)}
\end{align}
and
\begin{align}
						\label{a est 2}
\| \phi(\Delta)u\|_{L_p\left( [0,T] ; L_p\right)} \leq \bar N \|f\|_{L_p\left( [0,T] ; L_p\right)},
\end{align}
where $N=N(d,p,\delta_k,N_j,T)$ and 
$\bar N= \bar N(d,p,\delta_k,N_j)$ ($k=1,2,3$ and $j=1,2,3,4$).
\end{lemma}

\begin{proof}
Observe that by Fubini's theorem and Assumption \ref{as}(i),
\begin{align}
u(t,x) 
					\notag
&= \int_0^t \bE\left[  \cF^{-1}\left[ \cF\left[f(s,\cdot+X_t -X_s)\right](\xi) \right](x) \right]ds \\
						\notag
&= \int_0^t  \cF^{-1}\left[\bE\left[e^{i(\xi \cdot (X_t-X_s))}\right] \left[ \cF\left[f(s,\cdot)\right](\xi) \right](x) \right] ds\\
					\notag
&= \int_0^t  \cF^{-1}\left[\exp\left(\Phi_X(s,t,\xi)\right) \left[ \cF\left[f(s,\cdot)\right](\xi) \right](x) \right] ds\\
					\label{fo u}
&= \cF^{-1}\left[ 
\int_0^t  \exp\left(\Phi_X(s,t,\xi)\right) \left[ \cF\left[f(s,\cdot)\right](\xi) \right] ds 
 \right](x).
\end{align}
Recalling 
$\Phi_X(t,t,\xi)=0$,
by Assumption \ref{as}(i) again we have
\begin{align}
					\notag
\cA(t)u(t,x) &= \lim_{h \downarrow 0}\frac{\bE\left[u(t,x+X_{t+h}-X_t)-u(t,x)\right]}{h} \\
					\notag
&= \lim_{h \downarrow 0}\frac{ \bE \left[\cF^{-1}\left[ \left(e^{i\xi \cdot (X_{t+h} -X_t)}-1 \right)\cF[u(t,\cdot)](\xi)\right](x) \right]}{h} \\
					\notag
&= \lim_{h \downarrow 0}\cF^{-1}\left[  \frac{ \exp(\Phi_X(t,t+h,\xi))-\exp(\Phi_X(t,t,\xi))}{h}\cF[u(t,\cdot)](\xi)\right](x)  \\
					\label{f au}
&= \cF^{-1}\left[ \Psi_X(t,\xi)\cF[u(t,\cdot)](\xi)\right](x) \\
					\notag
&= \cF^{-1}\left[ \int_0^t  \Psi_X(t,\xi) \exp\left(\Phi_X(s,t,\xi)\right) \left[ \cF\left[f(s,\cdot)\right](\xi) \right] ds    \right](x) 
\end{align}
and
\begin{align}
					\notag
&\frac{\partial}{\partial t} \int_0^t  \exp\left(\Phi_X(s,t,\xi)\right) \left[ \cF\left[f(s,\cdot)\right](\xi) \right] ds \\
					\notag
&=\cF\left[f(t,\cdot)\right](\xi)
+ \int_0^t  \frac{\partial}{\partial t} \left(\exp\left(\Phi_X(s,t,\xi)\right) \right) \left[ \cF\left[f(s,\cdot)\right](\xi) \right] ds \\
					\label{910 1}
&=\cF\left[f(t,\cdot)\right](\xi)
+ \int_0^t  \Psi_X(t,\xi) \exp\left(\Phi_X(s,t,\xi)\right) \left[ \cF\left[f(s,\cdot)\right](\xi) \right] ds.
\end{align}
Since the last term above is integrable with respect to $\xi$ uniformly $t \in (0,T)$ for any $T\in (0,\infty)$,  we get
(\ref{a sol}) by taking the inverse Fourier transform to both sides of (\ref{910 1}). 

Next we show (\ref{a est 1}) and (\ref{a est 2}).
Due to the definition of $u$ and Minkowski's inequality,
\begin{align*}
\|u\|_{L_p([0,T];L_p)}
&= \left\| \int_0^t \bE\left[ f(s,x+X_t -X_s) \right]ds \right\|_{L_p([0,T];L_p)} \\
&\leq N(T) \| f\|_{L_p([0,T];L_p)}.
\end{align*}
Thus it suffices to show (\ref{a est 2}). 
We now prove this estimate in the following two steps.
\vspace{2mm}

\noindent
{\bf Step 1:} Assume $X =X^1$.  Note that if one takes $\psi=\phi$ and $\Psi=\Psi_{X^1}$ then 
 Assumption \ref{as}(ii) is exactly same as 
 Assumptions \ref{psi as 1}. 
Therefore due to (\ref{fo u}) and Theorem \ref{a p e psi},
\begin{align*}
\phi(\Delta)u (t,x)
= \cF^{-1}\left[ 
\phi(|\xi|^2)\int_0^t  \exp\left(\Phi_X(s,t,\xi)\right) \left[ \cF\left[f(s,\cdot)\right](\xi) \right] ds 
 \right](x)
\end{align*}
and
$$
\|\phi(\Delta)u\|_{L_p((0,T);L_p)}
\leq N \|f\|_{L_p((0,T);L_p)}.
$$

\noindent
{\bf Step 2 (General case):} Recall that two processes $X^1$ and $X^2$ are independent. 
Thus by Assumption \ref{as} and Fubini's theorem,

\begin{align*}
&\phi(\Delta)u(t,x)  
= \cF^{-1} \left[ \phi(|\xi|^2)\int_0^t \bE\left[ \cF\left[f(s,\cdot+X_t -X_s)\right](\xi) \right]ds \right](x)\\
&= \bE'\bigg[  \cF^{-1}\bigg( \phi(|\xi|^2) \times \\
&  \qquad \int_0^t \bE\left[ \cF\left[f(s,\cdot+X^1_t(\omega) -X^1_s(\omega)+X^2_t(\omega')-X^2_s(\omega'))\right](\xi)  \right]ds \bigg)(x) \bigg] \\
&= \bE'\bigg[  \cF^{-1}\bigg( \phi(|\xi|^2) \times \\
&  \qquad \int_0^t \bE\left[ \cF\left[f(s,\cdot+X^1_t(\omega) -X^1_s(\omega)-X^2_s(\omega'))\right](\xi)  \right]ds \bigg)
\left(x+X^2_t(\omega')\right) \bigg],
\end{align*}
where $\bE$ and $\bE'$ are the expectations with respect to the variables $\omega$  and $\omega'$, respectively. 
Since the paths of $X^2$ are locally bounded $(a.s.)$, one can easily check that
$f_{X^2}(s,x):=f(s,x-X_s^2)$ satisfies (\ref{e 1030 1}) $(a.s.)$.
For each fixed $\omega' \in \Omega$, denote 
$$
u_{X^2(\omega')}(t,x) := \int_0^t \bE\left[ f_{X^2(\omega')}(s,x+X^1_t -X^1_s) \right] ds.
$$
Then by Minkowski's inequality, the change of variable $x \to x-X^2_t(\omega')$ and the result of Step 1,
\begin{align*}
&\|\phi(\Delta)u\|_{L_p([0,T];L_p)}  
\leq 
\bE'\| \phi(\Delta)u_{X^2(\omega')}\|_{L_p([0,T];L_p)}\\
&\leq N\bE'\|f_{X^2(\omega')}\|_{L_p([0,T];L_p)} 
= N\|f\|_{L_p([0,T];L_p)}.
\end{align*}
The lemma is proved. 
\end{proof}
\begin{lemma}
						\label{uni lem}
Let $u \in C^\infty_c((0,T)\times \fR^d)$ and suppose that Assumption \ref{as} holds.
Then
\begin{align}
					\label{u f rela}
u(t,x) = \int_0^t \bE\left[ f(s,x+X_t -X_s) \right] ds\quad \forall (t,x) \in (0,T)\times \fR^d,
\end{align}
where 
$$
f(t,x)= u_t(t,x) - \cA(t)u(t,x).
$$
\end{lemma}
\begin{proof}
Recalling (\ref{f au}) and taking the Fourier transform, we have
\begin{align*}
\cF[f(t,\cdot)](\xi) = \frac{\partial}{\partial t}\cF[u(t,\cdot)](\xi)
- \Psi_X(t,\xi)\cF[u(t,\cdot)](\xi).
\end{align*}
For each $\xi$, solving the above ODE with respect to $t$, we have
\begin{align*}
\cF[u(t,\cdot)](\xi)
&= \int_0^t \exp\left( \Phi_X(s,t,\xi) \right) \cF[f(s,\cdot)](\xi) ds.
\end{align*}
Thus following (\ref{fo u}) in the reverse order, we obtain
(\ref{u f rela}) since the both sides of (\ref{u f rela}) are continuous on $(0,T)\times \fR^d$. 
The lemma is proved. 
\end{proof}

\vspace{3mm}
\noindent
{\bf Proof of Theorem \ref{main thm}}
\vspace{2mm}

\noindent
{\bf Step 1} (Existence)
\vspace{2mm}

Choose a sequence $f_n \in C_c((0,T) \times \fR^d)$ so that
$$
\|f_n - f\|_{L_p([0,T] ;L_p)} \to 0
$$
as $n \to \infty$. 
Define 
$$
u_n(t,x) := \int_0^t \bE\left[ f_n(s,x+X_t -X_s) \right].
$$
Then by (\ref{a est 1}) and (\ref{a est 2}),
\begin{align*}
\| u_n -u_m\|_{L_p\left( [0,T] ; H^\phi_{p}\right)} \leq N \|f_n -f_m\|_{L_p\left( [0,T] ; L_p\right)}
\end{align*}
and
\begin{align*}
\| \phi(\Delta)(u_n-u_m)\|_{L_p\left( [0,T] ; L_p\right)} \leq \bar N \|f_n-f_m\|_{L_p\left( [0,T] ; L_p\right)}.
\end{align*}
Since $L_p\left( [0,T] ; H^\phi_{p}\right)$ is a Banach space, 
$u_n$ converges to $u \in L_p\left( [0,T] ; H^\phi_{p}\right)$ and 
$u$ becomes a solution to equation (\ref{main eqn}) according to Definition \ref{def sol}
and obviously $u$ satisfies (\ref{main est 1}) and (\ref{main est 2}). 

\vspace{2mm}
\noindent
{\bf Step 2 } (Uniqueness)
\vspace{2mm}

Let $u$ and $v$ be solutions to equation (\ref{main eqn}). Then by Definition \ref{def sol}, one can find sequences $u_n \in C^\infty_c((0,T)\times \fR^d)$ and $v_n \in C^\infty_c((0,T)\times \fR^d)$ so that
$$
\frac{\partial u_n}{\partial t}- \cA(t)u_n \to f \quad \text{in} \quad L_p\left( [0,T] ; L_p\right),
$$
$$
u_n \to u \quad \text{in} \quad L_p\left( [0,T] ; H^\phi_{p}\right)
$$
and
$$
\frac{\partial v_n}{\partial t}- \cA(t)v_n \to f \quad \text{in} \quad L_p\left( [0,T] ; L_p\right),
$$
$$
v_n \to v \quad \text{in} \quad L_p\left( [0,T] ; H^\phi_{p}\right)
$$
as $n \to \infty$.
Denote
$$
f_n=\frac{\partial u_n}{\partial t}- \cA(t)u_n
$$
and
$$
g_n=\frac{\partial v_n}{\partial t}- \cA(t)v_n
$$
Then by Lemma \ref{uni lem},
$$
u_n(t,x) = \int_0^t \bE\left[ f_n(s,x+X_t -X_s) \right] ds
$$
and
$$
v_n(t,x) = \int_0^t \bE\left[ g_n(s,x+X_t -X_s) \right] ds.
$$
Since both $f_n$ and $g_n$ converge to $f$ in $L_p\left( [0,T] ; L_p\right)$,
we have
$
u=v.
$
The theorem is proved. \qed

\mysection{Appendix: Proof of Theorem \ref{fs thm}}
													 \label{hl fs pf}

Throughout this section, let $(O,\rF,\mu)$ be a complete measure space such  that
$$
\mu(O)=\infty.
$$
By $\rF_0$ we denote the subset of $\rF$ consisting of all sets $A$ such that $\mu(A) < \infty$.
$\bL(O,\rF,\mu)$ indicates the space of all locally integrable functions $f$ on $(O, \rF, \mu)$, i.e..
$$
f \in \bL(O,\rF,\mu) \quad  \Leftrightarrow \quad f1_A \in L_1(O, \rF,\mu) \quad \forall A \in \cF_0.
$$

If the given measure space is clear, we simply use notation $\bL$.
We borrow terminologies from \cite[Chapter 3]{Krylov2008}.
\begin{defn}
                \label{def filt}
We say that a collection $\rP \subset \rF_0$ is a partition if and only if elements of $\rP$ are countable, pairwise disjoint, and
$$
\bigcup_{\cP \in \rP} \cP = O.
$$
\end{defn}
\begin{remark}
Due to the definition of the partition,
the measure space $(O,\rF,\mu)$ is $\sigma$-finite if there is a partition $\rP$ on $(O,\rF,\mu)$.
\end{remark}

\begin{defn}
                    \label{filt defn}
Let $(\rP_n, n\in \bZ)$ be a sequence of partitions.
We say that $(\rP_n, n\in \bZ)$ is a filtration of partitions on $(O,\rF,\mu)$ if and only if

(i)
$$
\inf_{\cP \in \rP_n} \mu(\cP) \to \infty , \quad \text{as} \quad n \to -\infty
$$
and
\begin{align}
                    \label{ld pro}
\lim_{n\to \infty} \frac{1}{\mu(\cP_n(x))}\int_{\cP_n(x)}f(y)\mu(dy)=f(x) \quad (a.e.) \quad \forall f \in \bL,
\end{align}
where $\cP_n(x)$ denote the element of $\rP_n$ containing $x$;

(ii) For each $n \in \bZ$ and $\cP \in \rP_n$, there is a (unique) $\cP' \in \rP_{n-1}$ such that $\cP \subset \cP'$ and
$$
\mu(\cP') \leq N_0 \mu(\cP),
$$
where $N_0$ is a constant independent of $n$, $\cP$, and $\cP'$.
\end{defn}

%Examples of a filtration of partitions will be given later.
%Throughout this section, we assume that a filtration of partitions on a complete measure space $(U,\rF,\mu)$ denoted by $(\rP_n, n \in \bZ)$ is given.
We introduce a general Fefferman-Stein sharp function  related to the filtration of partition $(\rP_n, n\in \bZ)$. For a locally integrable function $f$ on $(O,\rF,\mu)$, we define its sharp function $f^\#$ 
as 
$$
f^\#(x) := \sup_{n \in \bZ} \aint |f(y)-f_{| n}(x)|\mu(dy):=\sup_{n \in \bZ}\frac{1}{\mu(\cP_n(x))} \int_{\cP_n(x)} |f(y) - f_{|n}(x)| \mu(dy),
$$
where
$$
f_{|n}(x):=\frac{1}{\mu(\cP_n(x))} \int_{\cP_n(x)} f(y) \mu(dy).
$$

At last, we introduce a version of Fefferman-Stein theorem on a measure space $(O, \rF, \mu)$ with a filtration.
\begin{thm}
                \label{kry fe-st}
For any $f \in L_p(O, \rF, \mu)$,
$$
\|f\|_{L_p(O,\rF,\mu)} \leq N\|f^\#\|_{L_p(O,\rF,\mu)},
$$
where $p \in (1,\infty)$, $q= p/(p-1)$, and $N= (2q)^pN_0^{p-1}$.
\end{thm}
\begin{proof}
See \cite[Lemma 3.2.4]{Krylov2008} and \cite[Theorem 3.2.10]{Krylov2008}.
\end{proof}

For $n, i_1,\cdots, i_d\in \bZ$, denote 
$$
\fB_{2^{-n}}(i_1,\ldots,i_d)=(i_12^{-n},(i_1+1)2^{-n}] \times \cdots \times (i_d 2^{-n},(i_d+1)2^{-n}].
$$
Recall $U=\fR^d$ or $U=\fR^d_+$.
Finally we construct a filtration on $(0,\infty)\times U$ related to the function $\varphi$ in Assumption \ref{as v}(i).

\begin{thm}
                \label{ex filt 3}
Suppose that Assumetion \ref{as v}(i) holds.
Then  there exists a sequence $(\sigma_n, n \in \bZ)$ such that  $\sigma_n \in [1,2)$, 
$$
\rP_n:=\left\{ \big(i\varphi(2^{-n})\sigma_n ,(i+1)\varphi(2^{-n})\sigma_n \big] \times \fB_{2^{-n}}(i_1,\ldots,i_d), \,\, i \in \bZ_+ ,i_1,\ldots,i_d \in \bZ \right\}
$$
and 
$$
\rP^{+}_{n}:=\left\{ \big(i\varphi(2^{-n})\sigma_n ,(i+1)\varphi(2^{-n})\sigma_n \big] \times \fB_{2^{-n}}(i_1,\ldots,i_d),
 \,\,  i, i_1 \in \bZ_+ ,i_2,\ldots,i_d \in \bZ \right\}
$$ 
become a filtration of partitions on ${(0,\infty)\times \fR^d}$ and ${(0,\infty)\times \fR^d_+}$ respectively.
\end{thm}

\begin{proof}
Because of similarity, we only construct the filtration $\rP_n$.
We construct this filtration in inductive ways.
Recall that $\varphi(r)$ is a nonnegative nondecreasing function from $(0,\infty)$ into $(0,\infty)$ so that
\begin{align}
                \label{psi con 1}
\varphi(r) \downarrow 0 \quad \text{as} \quad r \downarrow 0,\quad \varphi(r) \uparrow \infty \quad \text{as} \quad r \uparrow \infty,
\end{align}
and
\begin{align}
                \label{psi con 2}
\sup_{r>0} \frac{\varphi( 2 r)}{\varphi(r)} < \infty.
\end{align}
First, we set
$$
\rP_0 := \{ (i \varphi(1),(i+1)\varphi(1)] \times (i_1,i_1+1] \times \cdots \times (i_d,i_d+1], \quad i\in \bZ_+,i_1,\ldots,i_d \in \bZ \}
$$
and construct $\rP_n$ for $n=1, 2,\ldots$ inductively.
Suppose that $\rP_k$ is given for some $k \in \bZ_+$ and
$$
\rP_k=\{ \big(i\varphi(2^{-k})\sigma_k ,(i+1)\varphi(2^{-  k})\sigma_k \big] \times \fB_{2^{-k}}(i_1,\ldots,i_d), \quad i\in \bZ_+,i_1,\ldots,i_d \in \bZ \},
$$
where $\sigma_k \in [1,2)$ and
$$
\fB_{2^{-k}}(i_1,\ldots,i_d)=(i_12^{-k},(i_1+1)2^{-k}] \times \cdots \times (i_d 2^{-k},(i_d+1)2^{-k}].
$$
If $k=0$, then obviously $\sigma_k = 1$.
Since $\varphi$ is nondecreasing and $\varphi>0$, there exists a $\bZ_+$ so that
$$
\frac{\varphi(2^{-k}) \sigma_k}{\varphi(2^{-(k+1)})} \in [2^{\ell_{k+1}}, 2^{\ell_{k+1}+1}).
$$
We put
$$
\sigma_{k+1} = \frac{\varphi(  2^{-k}) \sigma_k}{\varphi(2^{- (k+1)})2^{\ell_{k+1}}}
$$
and define $\rP_{k+1}$ as the collection of sets
\begin{align*}
\big(i\varphi(2^{-(k+1)})\sigma_{k+1} ,(i+1)\varphi(2^{-(k+1)})\sigma_{k+1} \big] \times \fB_{2^{-(k+1)}}(i_1,\ldots,i_d),
\end{align*}
for all  $i\in \bZ_+,i_1,\ldots,i_d \in \bZ$. Then obviously
$$
\sigma_{k+1} \in [1,2)
$$
and for any $\cP \in \rP_{k+1}$ there exists a unique $\cP' \in \rP_k$ so that
\begin{align}
                \label{regul con 1}
\cP \subset \cP'
\end{align}
and
\begin{align*}
\frac{|\cP'|}{|\cP|}
= \frac{\varphi(2^{-k})\sigma_k}{\varphi(2^{-(k+1)})\sigma_{k+1}} 2^d
= 2^{d}2^{\ell_{k+1}} \leq 2^{d+1}\sup_{r>0} \frac{\varphi( 2 r)}{\varphi(r)} <\infty.
\end{align*}
In order to confirm \eqref{regul con 1}, observe that
if $\ell_{k+1}=0$ then for any $i \in \bZ$
\begin{align*}
\big(i\varphi(2^{-k})\sigma_k ,(i+1)\varphi(2^{-  k})\sigma_k \big] = \big(i\varphi(2^{-(k+1)})\sigma_{k+1} ,(i+1)\varphi(2^{-(k+1)})\sigma_{k+1} \big],
\end{align*}
 and on the other hand if $\ell_{k+1} > 0$ then
\begin{align*}
&\big(i\varphi(2^{-k})\sigma_k ,(i+1)\varphi(2^{-  k})\sigma_k \big]  \\
&= \bigcup_{l=0}^{2^{\ell_{k+1}}-1} \big(i_l \varphi(2^{-(k+1)})\sigma_{k+1} ,( i_l+1)\varphi(2^{-(k+1)})\sigma_{k+1} \big],
\end{align*}
where $i_l = i 2^{\ell_{k+1}}+l$. 

Next we construct $\rP_n$ for $n=-1,-2,\ldots$. Similarly,
suppose that $\rP_k$ is given for some $k \in \{0,-1,-2,\ldots\}$ and
$$
\rP_k=\{ \big(i\varphi(2^{-k})\sigma_k ,(i+1)\varphi(2^{-  k})\sigma_k \big] \times \fB_{2^{-k}}(i_1,\ldots,i_d), \quad i\in \bZ_+,i_1,\ldots,i_d \in \bZ \},
$$
where $\sigma_k \in [1,2)$ and
$$
\fB_{2^{-k}}(i_1,\ldots,i_d)=(i_12^{-k},(i_1+1)2^{-k}] \times \cdots \times (i_d 2^{-k},(i_d+1)2^{-k}].
$$
Since $\varphi$ is nondecreasing, $\varphi >0$, and $\sigma_k \in [1,2)$, there exists a $\ell_{k-1} \in \bN \cup \{0\}$ so that
$$
\frac{\varphi(2^{-k})\sigma_k}{\varphi(2^{-(k-1)})} \in [2^{-\ell_{k-1}}, 2^{-\ell_{k-1}+1}).
$$
We put
$$
\sigma_{k-1} = \frac{2^{\ell_{k-1}} \varphi(  2^{- k})  \sigma_k }{\varphi(2^{-(k-1)})}
$$
and define $\rP_{k-1}$ as the collection of sets
\begin{align*}
\big(i\varphi(2^{-(k-1)})\sigma_{k-1} ,(i+1)\varphi(2^{-(k-1)})\sigma_{k-1} \big] \times \fB_{2^{-(k-1)}}(i_1,\ldots,i_d),
\end{align*}
for all  $i\in \bZ_+,i_1,\ldots,i_d \in \bZ$.
Then obviously
$$
\sigma_{k-1} \in [1,2)
$$
and for any $\cP \in \rP_{k}$ there exists a unique $\cP' \in \rP_{k-1}$ so that
\begin{align}
                \label{reg p 2}
\cP \subset \cP'
\end{align}
and
\begin{align*}
\frac{|\cP'|}{|\cP|}
= \frac{\varphi(2^{-(k-1)})\sigma_{k-1}}{\varphi(2^{-k})\sigma_{k}} 2^d \leq 2^{d+\ell_{k-1}}
= 2^{d+1}\sup_{r>0} \frac{\varphi( 2 r)}{\varphi(r)} <\infty.
\end{align*}
\eqref{reg p 2} is due to the followings :
For any $i\in \bZ_+,i_1,\ldots,i_d \in \bZ$, if $\ell_{k-1}=0$ then
\begin{align*}
&\big(i\varphi(2^{-(k-1)})\sigma_{k-1} ,(i+1)\varphi(2^{-(k-1)})\sigma_{k-1} \big] &=  \big(i_l \varphi(2^{-k})\sigma_{k} ,( i_l+1)\varphi(2^{-k})\sigma_{k} \big] ,
\end{align*}
and on the other hand, unless $\ell_{k-1} = 0$ then
\begin{align*}
&\big(i\varphi(2^{-(k-1)})\sigma_{k-1} ,(i+1)\varphi(2^{-(k-1)})\sigma_{k-1} \big] \\
&= \bigcup_{l=0}^{2^{\ell_{k-1}}-1} \big(i_l \varphi(2^{-k})\sigma_{k} ,( i_l+1)\varphi(2^{-k})\sigma_{k} \big],
\end{align*}
where $i_l = i 2^{\ell_{k-1}} + l $. 

The theorem is proved.
\end{proof}
%\begin{remark}
%                \label{half filt}
%Following the proof of Theorem \ref{ex filt 3},
%one can easily construct a filtration on $(0,\infty) \times \fR^d_+$ 
%such as
%$$
%\rP_n:=\left\{ \big(i\varphi(2^{-n})\sigma_n ,(i+1)\varphi(2^{-n})\sigma_n \big] \times \fB_{2^{-n}}(i_1,\ldots,i_d), \quad i,i_1 \in \bZ_+ ,i_2,\ldots,i_d \in \bZ \right\}.
%$$
%\end{remark}

\vspace{2mm}
\noindent
{\bf Proof of Theorem \ref{fs thm}}
\vspace{2mm}

This is an easy consequence of Theorem \ref{kry fe-st} with the filtration
\begin{align*}
\rP_k
&=\left\{\left(i\varphi(2^{-k}) \sigma_k ,(i+1)\varphi(2^{-k}) \sigma_k \right] \times \fB_{2^{-k}}(i_1,\ldots,i_d) ~ : ~ i\in \bZ_+,i_1,\ldots,i_d \in \bZ \right\} \\
&=: \{ Q_{\varphi,k}(i,i_1,\ldots,i_d) ~ : ~ i\in \bZ_+,i_1,\ldots,i_d \in \bZ \}.
\end{align*}
We only remark that 
for any $Q_{\varphi,k}(i,i_1,\ldots,i_d) \in \cP$, one can find a $Q_c^\varphi(t_0,x_0) \in \bQ_\varphi$ so that
$$
Q_{\varphi,k}(i,i_1,\ldots,i_d)  \subset   Q_c^\varphi(t_0,x_0)
$$
and
$$
|Q_c^\varphi(t_0,x_0)|\leq N(d,\varphi)|Q_{\varphi,k}(i,i_1,\ldots,i_d)| .
$$
The theorem is proved. \qed

\mysection{acknowledgement} 
The authors are sincerely grateful to the anonymous referee for valuable suggestions and comments.

\end{document}